\documentclass[12pt,reqno]{amsart}%
\usepackage{float}
\usepackage[dvips]{graphics}
\usepackage{times}
\usepackage{mathrsfs}
\usepackage[T1]{fontenc}
\usepackage{latexsym}
\usepackage{epsfig}
\usepackage{hyperref}
\usepackage{breakurl}

\hypersetup{
    pdfborder = {0 0 0}
}
\usepackage{amsmath,amsfonts,amsthm,amssymb,amscd}
\usepackage{pstricks}
\usepackage[myheadings]{fullpage}
\newcommand{\Prob}[1]{{\rm Prob}\left(#1\right)}

\newcommand{\bburl}[1]{\textcolor{blue}{\url{#1}}}

\newcommand{\kommentar}[1]{}

\newcommand\be{\begin{equation}}
\newcommand\ee{\end{equation}}
\newcommand\bea{\begin{eqnarray}}
\newcommand\eea{\end{eqnarray}}
\newcommand\bi{\begin{itemize}}
\newcommand\ei{\end{itemize}}
\newcommand\ben{\begin{enumerate}}
\newcommand\een{\end{enumerate}}
\newcommand\bc{\begin{center}}
\newcommand\ec{\end{center}}
\newcommand\ba{\begin{array}}
\newcommand\ea{\end{array}}

\newtheorem{thm}{Theorem}[section]

\newtheorem{cor}[thm]{Corollary}
\newtheorem{lem}[thm]{Lemma}

\newtheorem{defi}[thm]{Definition}

\theoremstyle{definition}
\newtheorem{rek}[thm]{Remark}

\numberwithin{equation}{section}

\usepackage{placeins}
\usepackage{csquotes}
\usepackage{comment}

\title{Benford's Law Beyond Independence: Tracking Benford Behavior in Copula Models}
\author{Rebecca F. Durst}
\email{\textcolor{blue}{\href{mailto:rfd1@williams.edu}{rfd1@williams.edu}}}
\address{Department of Mathematics and Statistics, Williams College, Williamstown, MA 01267}
\curraddr{Department of Mathematics, Brown University, Providence, RI 02912}
\author{Steven J. Miller}\email{\textcolor{blue}{\href{mailto:sjm1@williams.edu}{sjm1@williams.edu}},  \textcolor{blue}{\href{Steven.Miller.MC.96@aya.yale.edu}{Steven.Miller.MC.96@aya.yale.edu}}}
\address{Department of Mathematics and Statistics, Williams College, Williamstown, MA 01267}

\begin{document}

\begin{abstract}
Benford's law describes a common phenomenon among many naturally occurring data sets and distributions in which the leading digits of the data are distributed with the probability of a first digit of $d$ base $B$ being $\log_{B}{\frac{d+1}{d}}$. As it often successfully detects fraud in medical trials, voting, science and finance, significant effort has been made to understand when and how distributions exhibit Benford behavior. Most of the previous work has been restricted to cases of independent variables, and little is known about situations involving dependence. We use copulas to investigate the Benford behavior of the product of $n$ dependent random variables. We develop a method for approximating the Benford behavior of a product of $n$ dependent random variables modeled by a copula distribution $C$ and quantify and bound a copula distribution's distance from Benford behavior. We then investigate the Benford behavior of various copulas under varying dependence parameters and number of marginals. Our investigations show that the convergence to Benford behavior seen with independent random variables as the number of variables in the product increases is not necessarily preserved when the variables are dependent and modeled by a copula. Furthermore, there is strong indication that the preservation of Benford behavior of the product of dependent random variables may be linked more to the structure of the copula than to the Benford behavior of the marginal distributions.

\end{abstract}

\maketitle

\tableofcontents

\section{Introduction}\label{sectionone}


Benford's law of digit bias applies to many commonly encountered data sets and distributions. A set of data $\{x_{i}\}_{i\in I}$ is said to be \textit{Benford base} $B$ if the probability of observing a value $x_{i}$ in the set with the first digit $d$ (where $d$ is any integer from $1$ to $B-1$) is given by the equation
\begin{equation}\label{eq: BenfProb}
\Prob{{\rm first \ digit \ of} \{x_{i}\}_{i \in I} \ {\rm is}\ d}\ {\rm base \ B} \ = \ \log_{B}\left(\frac{d+1}{d}\right).
\end{equation} These probabilities monotonically decrease, from base 10 about 30.103\% of the time having a leading digit of 1 to about 4.576\% of the time starting with a 9.

Benford's law was discovered in 1881 by the astronomer-mathematician Simon Newcomb who, looking at his logarithm table, observed earlier pages were more heavily worn than later pages. As logarithm tables are organized by leading digit, this led him to conclude that values with leading digit 1 occurred more commonly than values with higher leading digits. These observations were mostly forgotten for fifty years, when Benford \cite{Benford} published his work detailing similar biases in a variety of settings. Since then, the number of fields where Benford behavior is seen has rapidly grown, including accounting, biology, computer science, economics, mathematics, physics and psychology to name a few; see \cite{BH1, BH2, Miller, Nig1, Rai} for a development of the general theory and many applications. This prevalence of Benford's law, particularly in naturally occurring data sets and common distributions, has allowed it to become a useful tool in detecting fraud. One notable example of this was its in 2009 to find evidence suggesting the presence of fraud in the Iranian elections \cite{Battersby}; while Benford's law cannot prove that fraud happened, it is a useful tool for determining which sets of data are suspicious enough to merit further investigation (which is of great importance given finite resources); see for example \cite{Nig2, Singleton}.


To date, most of the work on the subject has involved independent random variables or deterministic processes (see though \cite{B--, IMS} for work on dependencies in partition problems). Our goal below is to explore dependent random variables through copulas, quantifying the connections between various relations and Benford behavior.

Copulas are multivariate probability distributions restricted to the unit hypercube by transforming the marginals into uniform random variables via the probability integral transform (see Section \ref{sec: sectiontwo} for precise statements). The term copulas was first defined by Abe Sklar in 1959, when he published what is now known as \textit{Sklar's Theorem} (see Theorem \ref{thm: sklar}), though similar objects were present in the work of Wassily Hoeffding as early as 1940. Sklar described their purpose as linking $n$-dimensional distributions with their one-dimensional margins. See \cite{Nelsen} for a detailed account of the presence and evolution of copulas.

Quoting the \textit{Encyclopedia of Statistical Sciences}, Nelsen \cite{Nelsen} writes: ``Copulas [are] of interest to statisticians for two main reasons: Firstly, as a way of studying scale-free measures of dependence; and secondly, as a starting point for constructing families of bivariate distributions, sometimes with a view to simulation.'' More specifically, copulas are widely used in application in fields such as economics and actuarial studies; for example, Kpanzou \cite{Kpanzou} describes applications in survival analysis and extreme value theory, and Wu \cite{Wu} details the use of Archimedean copulas in economic modeling and risk management. Thus, as copulas are a convenient and useful way to model dependent random variables, they are often employed in fields relating to finance and economics. Since many of these areas are also highly susceptible to fraud, it is worth exploring connections between copulas and Benford's law, with the goal to develop data integrity tests.

Essentially, since so many dependencies may be modeled through copulas, it is natural to ask when and how often these structures will display Benford behavior. In this paper, we investigate when data modeled by a copula is close to Benford's law by developing a method for approximating Benford behavior. In Section \ref{sec: sectionthree}, we develop this method for the product of $n$ random variables whose joint distribution are modeled by the copula $C$. We then apply this method in Section \ref{sec: sectionthreepointfive} to directly investigate Benford behavior for various copulas and dependence parameters. We conclude that Benford behavior depends heavily on the structure of the copula. Figure \ref{fig:intro} shows a small subset of the results covered in Section \ref{sec: sectionthreepointfive} that display the changing $\chi^{2}$ values of three different copulas compared to a Benford distribution as their dependence parameters vary. These three plots indicate how three different copulas modeling the same marginals may display drastically different behavior. Furthermore, Figure \ref{fig:intro2} shows the $\chi^{2}$ values of a copula compared to a Benford distribution as the number of marginals increases. As we will show, the behavior seen in this plot indicates that the product of many random variables with dependence modeled by a copula will not necessarily level-off like products of independent random variables, the log of which we may expect to become more uniform as the number of variables increases.

\begin{figure}[h]
\includegraphics[width=\linewidth]{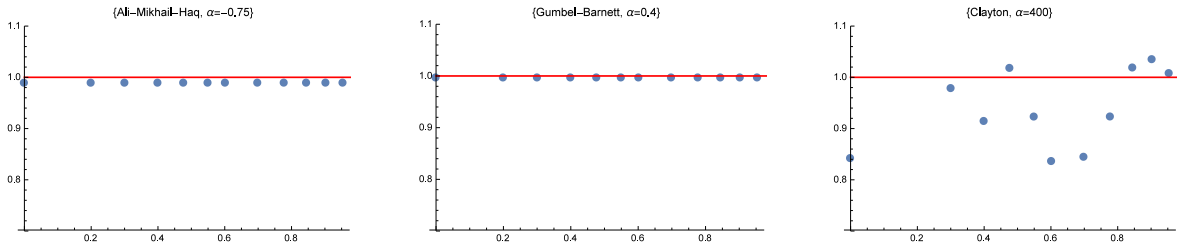}
\caption{The $y$-axes of these three plots represent the approximate values of the copula PDF of $\log_{10}{XY} \bmod 1$ at various values of $x \in [0,1]$, where $X$ and $Y$ are the marginal distributions. In each case, the marginal distributions are $\rm{N}(0,1)$, and $\rm{Pareto}(1)$. A Benford variable is equivalent to $\log_{10}{XY} \bmod 1 = 1$ at all points, thus the red line at $y=1$ represents the PDF of a Benford variable. We clearly see that the first two cases are close to Benford at all points. The third case, however, is highly variable and therefore does not display Benford behavior.} \label{fig:intro}
\end{figure}
\FloatBarrier

\begin{figure}[h]
\includegraphics[width=0.4\linewidth]{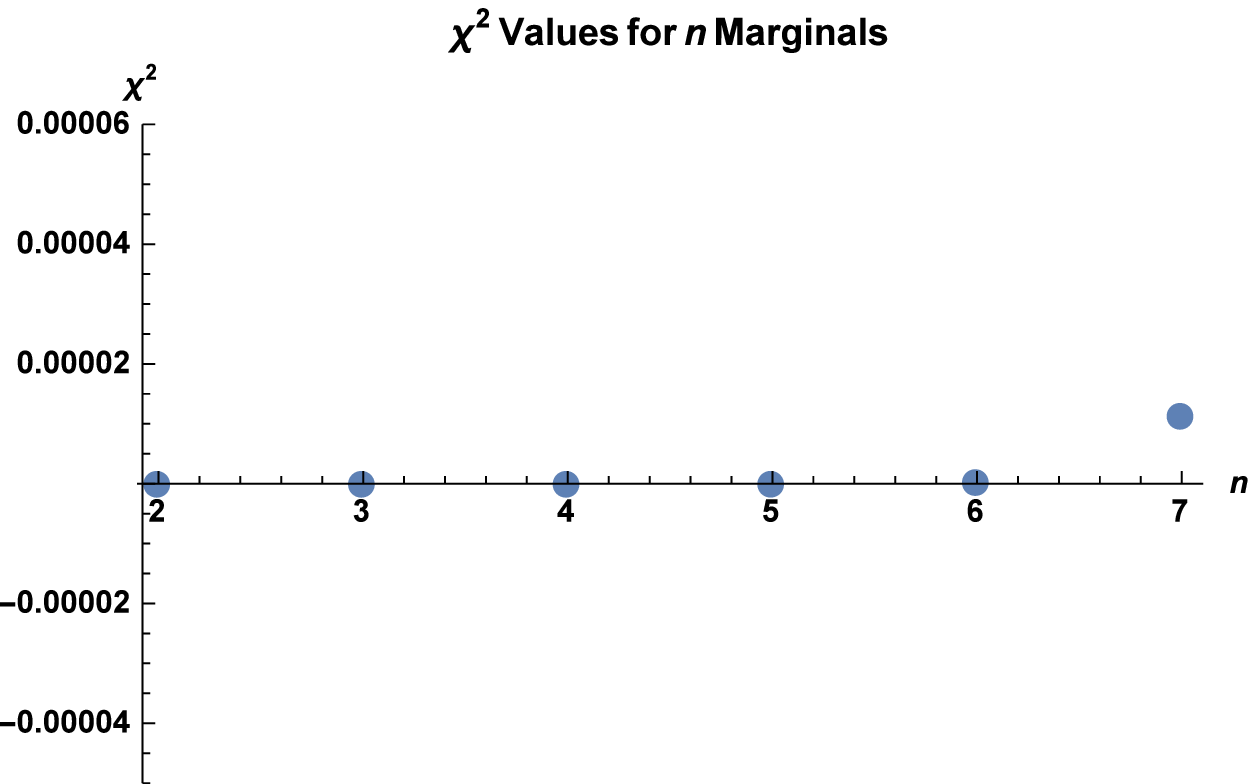}
\caption{The $\chi^{2}$ values comparing the behavior of the product to a Benford PDF as the number of marginals increases.  We have $8$ degrees of freedom and a significance level of $0.005$, so we reject the hypothesis if the value exceeds $1.3$.}\label{fig:intro2}
\end{figure}
\FloatBarrier

The results of this paper extend current techniques for testing Benford's law to situations where independence is not guaranteed, allowing analyses like that carried out by Cuff et. al. \cite{Cuff} on the Weibull and Durst et. al. \cite{D--} on the Inverse Gamma distributions to be conducted in the case of $n$ dependent random variables. In Section \ref{sec: sectionfour}, we restrict ourselves to $n$-tuples of random variables in which at least one is a Benford distribution and develop a concept of distance between our joint distribution and a Benford distribution, thus developing a concept of distance from a Benford distribution in order to understand how much deviation from Benford one might expect of a particular distribution. We then provide an upper bound for this distance using the $L^{1}$ norm of the function $N(u_{1},u_{2},\dots,u_{n}) = 1-\frac{\partial^{n}C(u_{1},u_{2},\dots,u_{n})}{\partial{u_1}\partial{u_2}\dots \partial{u_n}}$. In doing so, we draw an interesting connection between the distance from a Benford distribution and a copula's distance from the space of copulas for which $C_{uv}(u,v)=1$ for all $u,v$ in $[0,1]$.

\section{Terms and Definitions}\label{sec: sectiontwo}

We abbreviate \textit{probability density functions} by PDFs and \textit{cumulative distribution functions} as CDFs, and assume all CDFs are uniformly or absolutely continuous. All results below are standard; see the references for proofs.

\subsection{General Mathematics and Benford's Law}

\begin{lem}[Barbalat's Lemma (Lemma 2.1 in \cite{Fontes})]\label{thm: barbalat}
Let $t \to F(t)$ be a differentiable function with a finite limit as $t \to \infty$. If $F'$ is uniformly continuous, then $F'(t) \to 0$ as $t \to \infty$.
\end{lem}

\begin{defi}[Scientific Notation]
Any real number, $x$, can be written in the form
\begin{equation}
x \ = \ S_{B}(x) \cdot B^{n},
\end{equation}
where $n$ is an integer. We call $B$ the \textbf{base} and $S_{B}(x)$ the \textbf{significand}.
\end{defi}

We define strong Benford's law base $B$ (see, for example, \cite{BH2, Miller}). This is the definition we primarily use in Section \ref{sec: sectionthree}; strong indicates that we are studying the entire significand of the number and not just its first digit. In Section \ref{sec: sectionfour}, we will provide insight into how one may define a weaker version of Benford's law that permits the probabilities to be within $\epsilon$ of the theoretical Benford probabilities.

\begin{defi}[Strong Benford's Law (see  Definition 1.6.1 of \cite{Miller})]
A data set satisfies the Strong Benford's Law Base $B$ if the probability of observing a leading digit of at most $s$ in base $B$ is $\log_{B}{s}$.
\end{defi}

\begin{thm}[Absorptive Property of Benford's Law (see page 56 of \cite{Tao}).]\label{thm:benfabs}
Let $X$ and $Y$ be \textit{independent} random variables. If $X$ obeys Benford's law, then the product $W = XY$ obeys Benford's law regardless of whether or not $Y$ obeys Benford's law.
\end{thm}

\subsection{Copulas}
All theorems and definitions in this section are from Nelsen \cite{Nelsen} unless otherwise stated.

\begin{rek} In \cite{Nelsen}, functions are defined on the \textbf{extended real line}, $[-\infty,\infty]$; thus $f(t)$ is defined when $t=\pm \infty$. We use this notation in order to maintain consistency with \cite{Nelsen}, as this is one of the central texts in copula theory.
\end{rek}

\begin{defi}[$n$-Dimensional Copula]
An \textbf{n-dimensional copula}, $C$, is a function satisfying the following properties:
\begin{enumerate}
\item the domain of $C$ is $[0,1]^{n}$,
\\
\item ($n$-increasing) the $n$\textsuperscript{th}-order difference of $C$ is greater than or equal to zero,
\\
\item (grounded) $C(u_{1},u_{2},\dots,u_{n})=0$ if $u_{k}=0$ for at least one $k$ in $\{1,2,\dots,n\}$, and
\\
\item $C(1,1,\dots,1,u_{k},1,,\dots,1)=u_{k}$ for some $k$ in $\{1,2,\dots,n\}$.
\end{enumerate}
\end{defi}

\begin{thm}[Sklar's Theorm (Theorem 2.10.9 in \cite{Nelsen})]\label{thm: sklar}
Let $H$ be a $n$-dimensional distribution function with marginal CDFs $F_{1}, F_{2}, \dots,F_{n}$. Then there exists a $n$-copula $C$ such that for all $(x_{1},x_{2},\dots,x_{n})$ in $[-\infty,\infty]^{n}$,

\begin{equation}\label{thm: forSklar}
H(x_{1},x_{2},\dots,x_{n}) \ = \ C(F_{1}(x_{1}),f_{2}(x_{2}),\dots,F_{n}(x_{n})).
\end{equation}
If all $F_{i}$ are continuous, then $C$ is unique; otherwise, $C$ is uniquely determined on ${\rm Range}(F_{1}) \times {\rm Range}(F_{2})\times \cdots \times {\rm Range}(F_{n})$. Conversely, if $C$ is a copula and $F_{1},F_{2},\dots,F_{n}$  are cumulative distribution functions, then the function $H$ defined by (\ref{thm: forSklar}) is a  distribution function with marginal cumulative distribution functions $F_{1},F_{2},\dots,F_{n}$.
\end{thm}

\begin{thm}[Extension of Theorem 2.4.2 in \cite{Nelsen}]
Let $X_{1},X_{2},\dots,X_{n}$ be continuous random variables. Then they are independent if and only if their copula, $C_{X_{1},X_{2},\dots,X_{n}}$, is given by $C_{X_{1},X_{2},\dots,X_{n}}(x_{1},x_{2},\dots,x_{n})=\Pi(x_{1},x_{2},\dots,x_{n}) =x_{1}x_{2}\cdots x_{n})$, where $\Pi$ is called the \textit{product copula}.
\end{thm}

\begin{thm}[Extension of Theorem 2.4.3 in \cite{Nelsen}]\label{thm:noChange}
Let $X_{1},X_{2},\dots,X_{n}$ be continuous random variables with copula $C_{X_{1},X_{2}\dots X_{n}}$ . If $a_{1},a_{2},\dots,a_{n}$  are strictly increasing on $\rm{Range}(X_{1})$, $\rm{Range}(X_{2})$, $\dots$, $\rm{Range}(X_{n})$, respectively, then $C_{a_{1}(X_{1}),a_{2}(X_{2})\dots a_{n}(X_{n})}=C_{X_{1},X_{2}\dots X_{n}}$. Thus $C_{X_{1},X_{2}\dots X_{n}}$ is invariant under strictly increasing transformations of  $X_{1},X_{2},\dots,X_{n}$.
\end{thm}

\begin{rek}
For the following three definitions, see \cite{Nelsen} page 116 for the 2-copula formulas and page 151 for the n-copula extension.
\end{rek}

\begin{defi}[Clayton Family of Copulas]\label{def:clayton}
A (n-dimensional) copula in the \textbf{Clayton family} is given by the equation
\begin{equation}
C(u_{1},u_{2},\dots,u_{n}) \ = \ \max{\{u_{1}^{-\alpha}+u_{2}^{-\alpha}+\dots+u_{n}^{-\alpha}+n-1,0\}},
\end{equation}
where $\alpha \in [-1,\infty) \setminus \{0\}$ is a parameter related to dependence with $\alpha=0$ as the independence case.
\end{defi}

\begin{defi}[Ali-Mikhail-Haq Family of Copulas]\label{def:amh}
A (n-dimensional) copula in the \textbf{Ali-Mikhail-Haq family} is given by the equation
\begin{equation}C(u_{1},u_{2},\dots,u_{n}) \ = \ \frac{(1-\alpha)}{(\prod_{i=1}^{n}\frac{1-\alpha(1-u_{i})}{u_{i}})-\alpha},
\end{equation}
where $\alpha \in [-1,1)$ is a parameter related to dependence with $\alpha=0$ as the independence case.
\end{defi}

\begin{defi}[Gumbel-Barnett Family of Copulas]\label{def:gb}
A (n-dimensional) copula in the \textbf{Gumbel-Barnett family} is given by the equation
\begin{equation}C(u_{1},u_{2},\dots,u_{n}) \ = \ \exp{\frac{1 + (\alpha \log{u_{1}}-1)(\alpha \log{u_{2}}-1)\cdots (\alpha \log{u_{n}}-1)}{\alpha}},
\end{equation}
where $\alpha \in (0,1]$ is a parameter related to dependence with $\alpha =0$ as the independence case.
\end{defi}

\section{Testing for Benford Behavior of a Product}\label{sec: sectionthree}

We state the results below in arbitrary dimensions but for notational convenience give the proofs for just two dimensions (as the generalization is straightforward).


Let $X_{1},X_{2},\dots,X_{n}$ be continuous random variables with CDFs $F_{X_1}(x_1),\dots,F_{X_n}(x_n)$. Let their joint PDF be $H_{X_{1},X_{2},\dots,X_{n}}(X_{1},X_{2},\dots,X_{n})$. By Theorem (\ref{thm: sklar}), we know there exists a copula $C$ such that
\begin{equation}
H_{X_{1},X_{2},\dots,X_{n}}(X_{1},X_{2},\dots,X_{n}) \ = \ C(F_{X_1}(X_{1}),\dots,F_{X_{n}}(X_n)).
\end{equation}
Assume $X_{1},\dots,X_{n}$ are such that their copula $C$ is \textit{absolutely continuous}. This allows us to define the joint probability density function (see \cite{Nelsen}, page 27) by $\frac{\partial^{n}C}{\partial{x_1}\partial{x_2} \cdots \partial{x_n}}$ . Furthermore, we restrict ourselves to $X_i$ such that all $F_{X_i}$ are uniformly continuous, as this allows us to use Lemma (\ref{thm: barbalat}) to later ensure that the  PDFs approach zero in their right and left end limits.

From here we have the following lemma.

\begin{lem}\label{thm: benfsolve}
Given $X_{1},X_{2},\dots,X_{n}$ continuous random variables with joint distribution modeled by the absolutely continuous copula $C$, let $U_{i}=\log_{B}{X_{i}}$ for all $i \leq n$ and for some base $B$, and let the CDFs of each $U_{i}$ be $F_{i}(u_{i})$. Also, let $f_{i}(u_{i})$ be the PDF of $U_{i}$ for all $i$. Finally, let $\textbf{u}_{0}=(u_{1},u_{2},\dots,u_{n-1},s+k-(u_{1}+u_{2}+\cdots+u_{n-1}))$. Then
\begin{align} \label{eq: pdfmod11}
\nonumber& \Prob{(\sum_{i=1}^{n}U_{i})\hspace{0.5mm}\rm{mod}\hspace{0.5mm}1 \leq s} \ = \\
&\int_{0}^{s}  \sum_{k=-\infty}^{\infty}\int_{u_{1}=-\infty}^{\infty} \cdots \int_{u_{n-1}=-\infty}^{\infty}\frac{\partial^{n}C(F_{1}(u_{1}),\dots,F_{n-1}(u_{n-1}),F_{n}(u_{n}))}{\partial{u_1}\partial{u_2}\cdots \partial{u_n}}\Big|_{\textbf{u}_{0}}{du_{1}\cdots du_{n-1}}.
\end{align}
Therefore, the PDF of $(U+V)\hspace{0.5mm}\rm{mod}\hspace{0.5mm}1$ is given by
\begin{equation}\label{eq:pdfmod1}
\sum_{k=-\infty}^{\infty}\int_{u_{1}=-\infty}^{\infty} \cdots \int_{u_{n-1}=-\infty}^{\infty}\frac{\partial^{n}C(F_{1}(u_{1}),\dots,F_{n-1}(u_{n-1}),F_{n}(u_{n}))}{\partial{u_1}\partial{u_2}\cdots \partial{u_n}}\Big|_{\textbf{u}_{0}}{du_{1}\cdots du_{n-1}}.
\end{equation}
\end{lem}

See Appendix \ref{sec:Appendix 1} for the proof.\\

If (\ref{eq:pdfmod1}) equals $1$ for all $s$, then our product is Benford. If it is not identically equal to $1$ for all $s$, then at each point we may assign a value $\epsilon_{s}$ that represents our distance from a Benford distribution. Thus we have
\begin{equation}\label{eq:eps}
\epsilon_{s} \ = \ \left| 1 \ - \ \sum_{k=-\infty}^{\infty}\int_{u_{1}=-\infty}^{\infty} \cdots \int_{u_{n-1}=-\infty}^{\infty}\frac{\partial^{n}C(F_{1}(u_{1}),\dots,F_{n-1}(u_{n-1}),F_{n}(u_{n}))}{\partial{u_1}\partial{u_2}\cdots \partial{u_n}}\Big|_{\textbf{u}_{0}}{du_{1}\cdots du_{n-1}}  \right|.
\end{equation}
This formulation will form the basis of Section \ref{sec: sectionfour}.

Unfortunately, the infinite sum and improper integral in \eqref{eq:pdfmod1} makes it highly impractical to use in application unless we can determine a method to closely approximate it by a finite sum and finite integral. We note that \eqref{eq:pdfmod1} is a PDF, and so is $\frac{\partial^{n}C}{\partial{x_1}\partial{x_2}\cdots \partial{x_n}}$, so we have the following properties (for notational convenience we state them in the two-dimensional case; similar results hold for $n$-dimensions).
\begin{enumerate}
\item $\int_{0}^{1}\left(\sum_{k=-\infty}^{\infty}\int_{-\infty}^{\infty}C_{u_{1}u_{2}}(F_{1}(u_1),F_{2}(s+k-u_{1}))f_{1}(u_{1})f_{2}(s+k-u_{1})du_{1} \right)ds = 1$.
\\
\item $\sum_{k=-\infty}^{\infty}\int_{-\infty}^{\infty}C_{u_{1}u_{2}}(F_{1}(u_1),F_{2}(s+k-u_{1}))f_{1}(u_{1})f_{2}(s+k-u_{1})du_{1} \geq 0$ for all $s$.
\\
\item $\int_{-\infty}^{\infty}C_{u_{1}u_{2}}(F_{1}(u_1),F_{2}(s+k-u_{1}))f_{1}(u_{1})f_{2}(s+k-u_{1})du_{1} \to 0$ as $k \to \pm \infty$.
\\
\item $C_{u_{1}u_{2}}(F_{1}(u_1),F_{2}(s+k-u_{1}))f_{1}(u_{1})f_{2}(s+k-u_{1}) \to 0$ as $u_{1} \to \pm \infty$.
\end{enumerate}

Property ($1$) is simply the definition of a PDF, and Property ($2$) is a direct result of the fact that a PDF is always positive. Properties $3$ and $4$ are required, under Lemma \ref{thm: barbalat}, by the convergence of the integral in Property ($1$) and by the convergence of the sum.

From Properties ($3$) and ($4$) and the definition of convergence we obtain the following.

\begin{lem} [Approximating the PDF] \label{lem:convergence}
Given $U_{1},\dots,U_{n}$ continuous random variables modeled by the copula $C$ with marginal CDFs $F_{1},\dots,F_{n}$ and PDFs $f_{1},\dots,f_{n}$, then there exist $a_{1},\dots,a_{n-1}$, $b_{1},\dots,b_{n-1}$, and $c_{1}$ and $c_{2}$ completely dependent on the $F_{i}$ such that $a_{i}< b_{i}$ for all $i$ and $c_{1}<c_{2}$ and
\begin{align}\label{eq:boundedInt}
\nonumber &\sum_{k=-\infty}^{\infty}\int_{u_{1}=-\infty}^{\infty}\cdots  \int_{u_{n-1}=-\infty}^{\infty}\frac{\partial^{n}C(F_{1}(u_{1}),\dots,F_{n-1}(u_{n-1}),F_{n}(u_{n}))}{\partial{u_1}\partial{u_2}\cdots \partial{u_n}}\Big|_{\textbf{u}_{0}}{du_{1}\cdots du_{n-1}} \\
&\ = \ \sum_{k=c_{1}}^{c_{2}}\int_{u_{1}=a_{1}}^{b_{1}}\cdots  \int_{u_{n-1}=a_{n-1}}^{b_{n-1}}\frac{\partial^{n}C(F_{1}(u_{1}),\dots,F_{n-1}(u_{n-1}),F_{n}(u_{n}))}{\partial{u_1}\partial{u_2}\cdots \partial{u_n}}\Big|_{\textbf{u}_{0}}{du_{1}\cdots du_{n-1}} + \ E_{a,b,c}(s)
\end{align}
where $E_{a,b,c}(s) \to 0$ as each $a_{i}$ and $c_{1}$ go to $ -\infty$ and each $b_{i}$ and $c_{2}$ go to $\infty$. Thus, for any $\epsilon > 0$, there exists (for each $i$) $|a_{i}|$, $|b_{i}|$, $|c_{1}|$, and $|c_{2}|$ large enough such that $|E_{a,b,c}(s)| \leq \epsilon$.
\end{lem}

The proof of this claim can be found in Appendix {\ref{sec:Appendix 1}}.


Because $s$ only ranges from $0$ to $1$, we can always find a value of $s$ that maximizes $E_{a,b,c}$ for any given set of $a$, $b$, and $c$ and set this to be the maximum error.  Furthermore, since all $f_{i}$ should have similar tail-end behavior, we do not have to worry about the divergence of one canceling out the divergence of the other. Thus, for this analysis to work, it is sufficient to understand the tail-end behavior of only one of the marginals.

In Appendix \ref{sec:Appendix 2}, we provide several examples of this method for testing for Benford behavior computationally with two variables.

\section{Testing For Benford Behavior: Examples}\label{sec: sectionthreepointfive}

 Now that an effective method for testing the Benford behavior of copulas has been established, we investigate how this behavior varies for specific copulas and marginals. In all $\chi^{2}$ tests, we have $11$ degrees of freedom and a significance level of $0.005$, so we reject the hypothesis if the value exceeds $2.6$. Our main interest, however, is to observe the how and if these values trend towards this critical value.
\subsection{ 2-Copulas with varying dependence parameter}
The following figures display the non-error values of \eqref{eq:boundedInt} at various values of $s$ for three different copulas. The red line in each plot indicates the constant function $y=1$ which will be achieved if the product $XY$ is exactly Benford. For each copula, we test three different pairings of marginals: $(A)$ $X~10^{\rm{N}(0,1)}$ and $Y~10^{\rm{Exp}(1)}$, $(B)$ $X~10^{\rm{Pareto}(1)}$ and $Y~10^{\rm{N}(0,1)}$, and $(C)$  $X~10^{\rm{Pareto}(1)}$ and $Y~10^{\rm{Exp}(1)}$. In each case, we vary the dependence parameter, $\alpha$ and compare the results to the case of independence. Our Pareto distribution has scale parameter $x_{m}=1$ and shape parameter $\alpha_{p}=2$. We note that in some cases the axes must be adjusted to be able to show any change in the Benford behavior.
\par
\subsubsection{Ali-Mikhail-Haq Copula:}
 Considering the independence case, $\alpha=0$, in Figure \ref{fig:amh1} we note that marginal pairings $(A)$ and $(B)$ have an approximately Benford product when independent. Pairing $(C)$, however, does not. From these plots, it is evident that the Ali-Mikhail-Haq copula displays notably consistent Benford behavior, as each plot remains very close to the independence case as $\alpha$ moves over its full range. This is reinforced by the corresponding plots in Figure \ref{fig:chi1}, which display the $\chi^{2}$ values of each marginal pairing for each value of alpha. We point out that although each plot indicates a general trend away from Benford behavior (the constant function $1$), the values for pairing $(A)$ are all smaller than $10^{-7}$, making them effectively $0$. Similarly, the values for pairing $(B)$ appear to increase linearly, but they are all of order of $10^{-6}$. The values for pairing $B$ vary from order $10^{-2}$ to order $10^{-1}$, suggesting that the behavior is both significantly less Benford and more variable than the other two pairings.
\FloatBarrier
\begin{figure}
\begin{minipage}[t]{0.31\linewidth}
\includegraphics[width=\linewidth,height=1.4\linewidth]{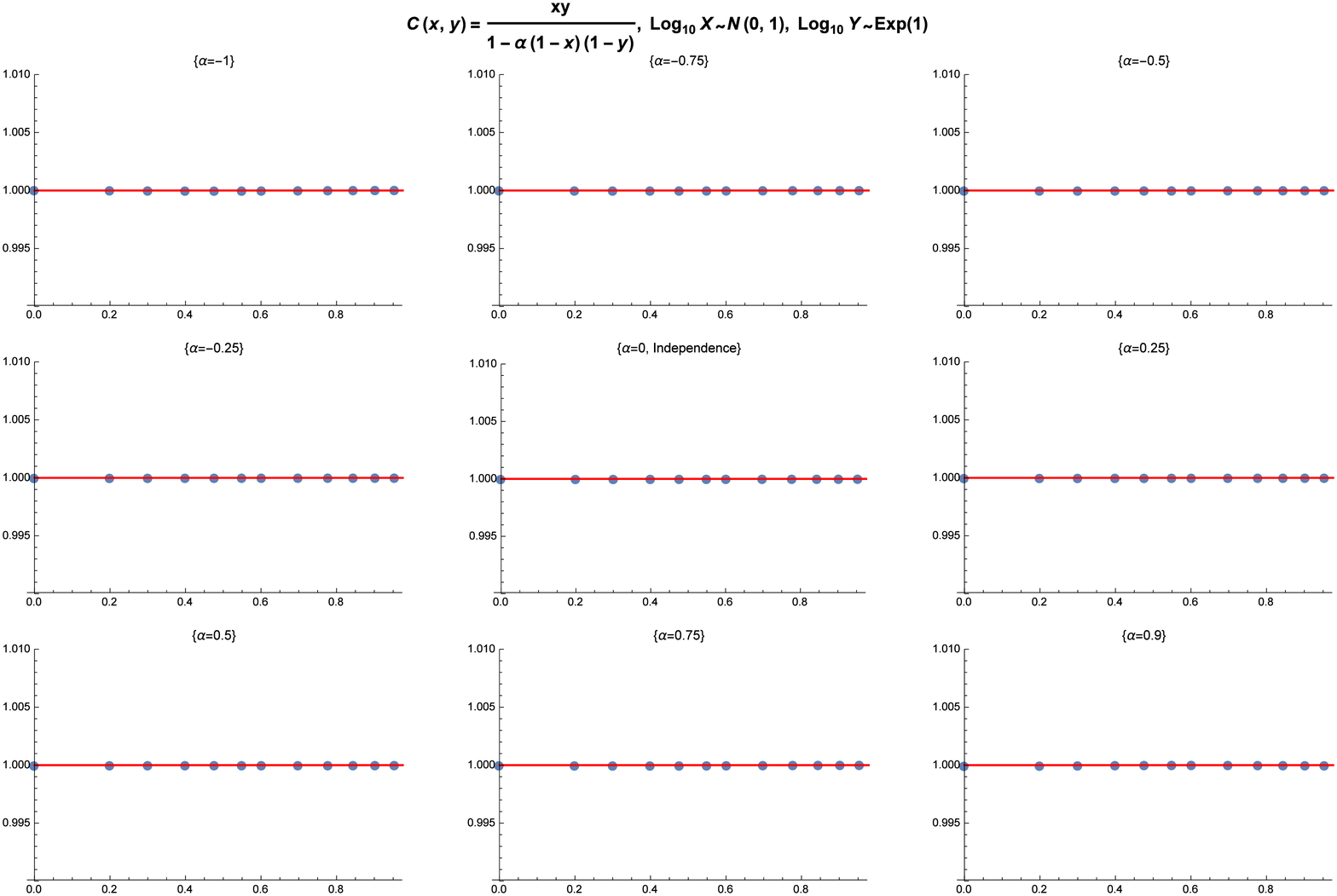}
\end{minipage}
\begin{minipage}[t]{0.31\linewidth}
\includegraphics[width=\linewidth,height=1.4\linewidth]{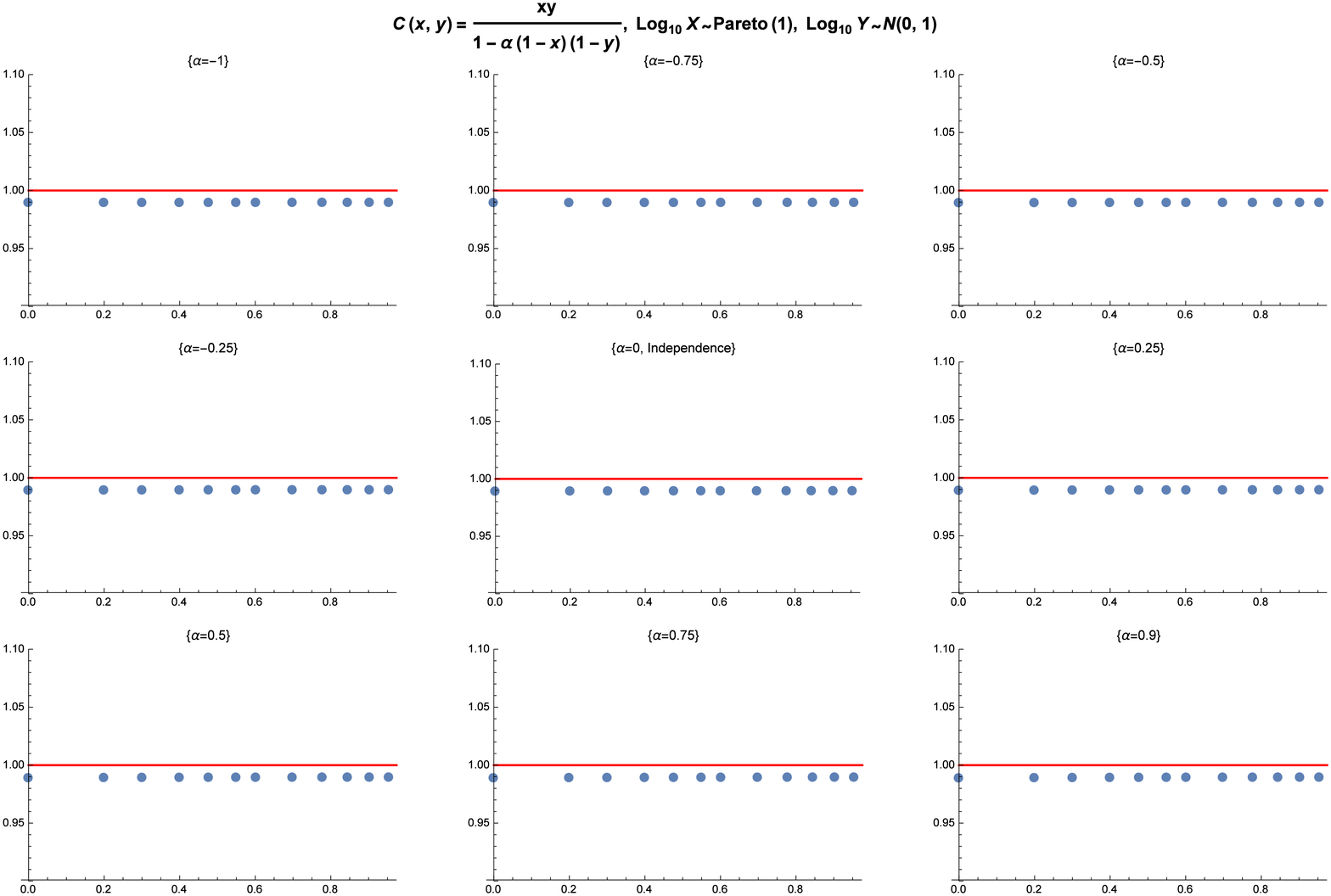}
\end{minipage}
\begin{minipage}[t]{0.31\linewidth}
\includegraphics[width=\linewidth,height=1.4\linewidth]{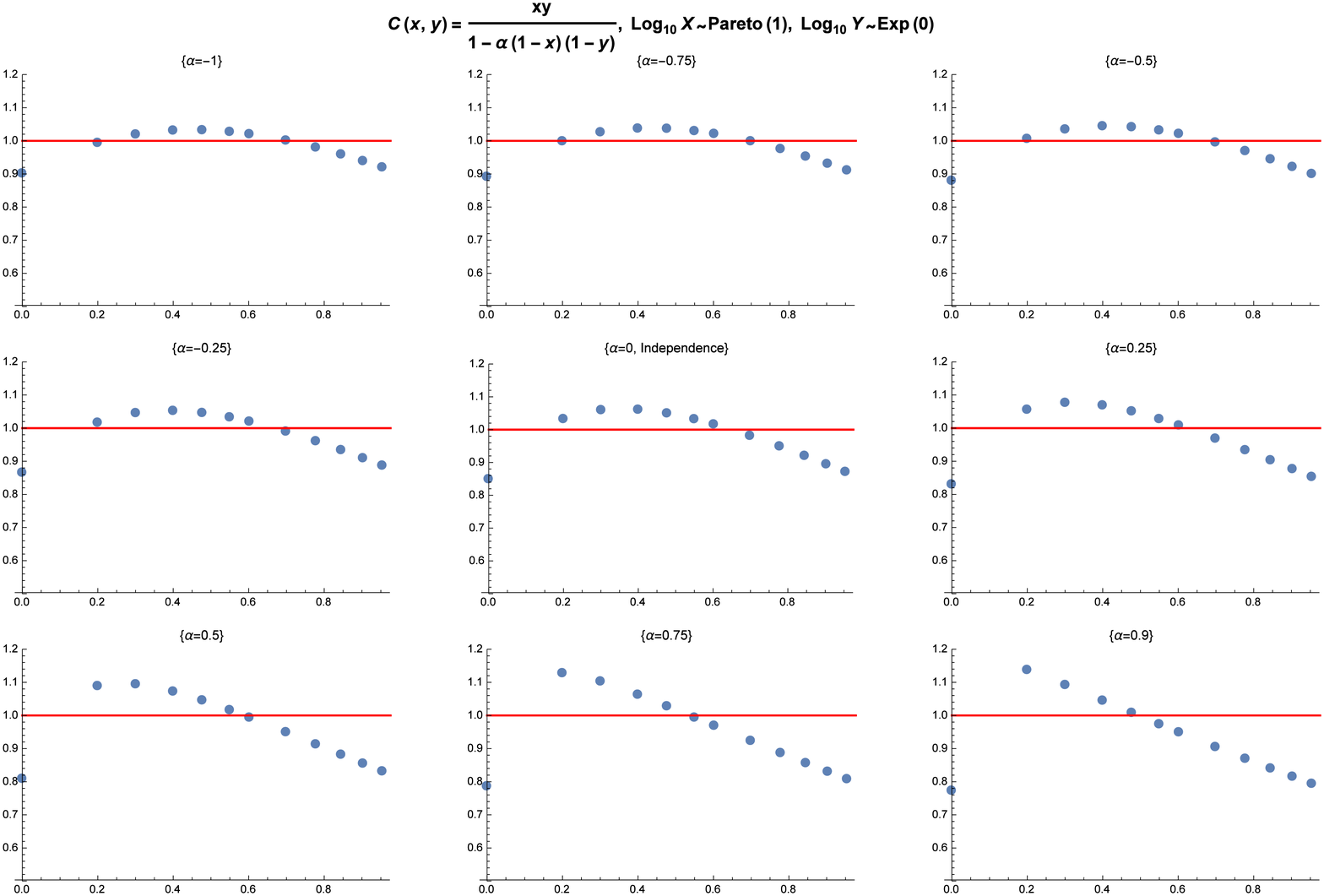}
\end{minipage}
\caption{The Ali-Mikhail-Haq 2-copula (see Definition \ref{def:amh}) modeled on three different sets of marginals with varying dependence parameter $\alpha \in [-1,1)$. The $y$-axes of these plots represent the approximate values of the copula PDF of $\log_{10}{XY} \bmod 1$ at various values of $x \in [0,1]$, where $X$ and $Y$ are the marginal distributions. The red line represents the Benford distribution.}\label{fig:amh1}
\end{figure}
\FloatBarrier

\begin{figure}[h]
\includegraphics[width=\linewidth]{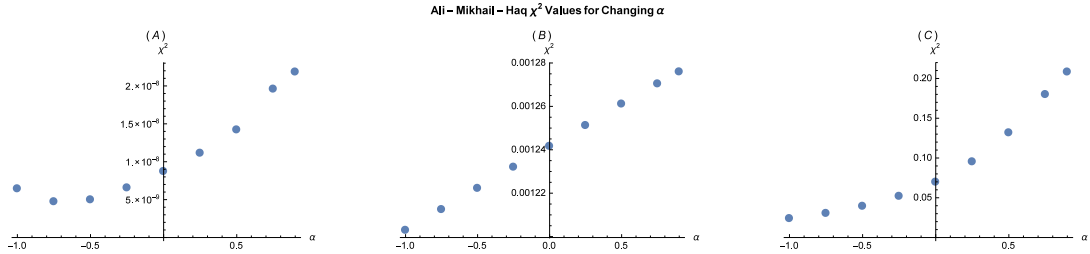}
\caption{The $\chi^{2}$ values associated to the the preceding sets of plots for the Ali-Mihkail-Haq copula. Each shows the comparison to Benford behavior as $\alpha$ increases. We have $11$ degrees of freedom and a significance level of $0.005$, so we reject the hypothesis if the value exceeds $2.6$. Clearly, only case $(C)$ comes close to rejecting the hypothesis.} \label{fig:chi1}
\end{figure}
\FloatBarrier

\par

\par
\subsubsection{Gumbel-Barnett Copula:}
These plots suggest that the Gumbel-Barnett copula undergoes even less change over $\alpha$ than the Ali-Mikhail-Haq copula. For pairings $(A)$ and $(B)$, the range for the plots must be restricted to $[0.9999,1.0001]$ and $[0.995,1.010]$, respectively, in order to show any change at all. Pairing $(C)$ is not nearly Benford, so its range is expected to vary (recall that the function described by each plot should integrate to 1 in the continuous case). We note, however, that the value at $s=0$ in pairing $(C)$ appears to vary over a range of $0.1$ as $\alpha$ increases. The $\chi^{2}$ plots in Figure \ref{fig:chi2} reinforce this interpretation, as in each case the values vary over a significantly small range.
\par
This lack of variation is likely due to the actual formula of the copula,
\begin{equation}\label{eq:gumbelbarnett}
C(x,y) \ = \ xy e^{-\alpha xy}
\end{equation}
 In this case, we have the independence copula, $C(x,y)=xy$ multiplied by a monotonic transformation of the independence copula, $e^{-a xy}$. Thus, it is possible that one or both of these elements serves to preserve the Benford properties of the marginals.
\FloatBarrier
\begin{figure}
\begin{minipage}[t]{0.32\linewidth}
\includegraphics[width=\linewidth,height=1.5\linewidth]{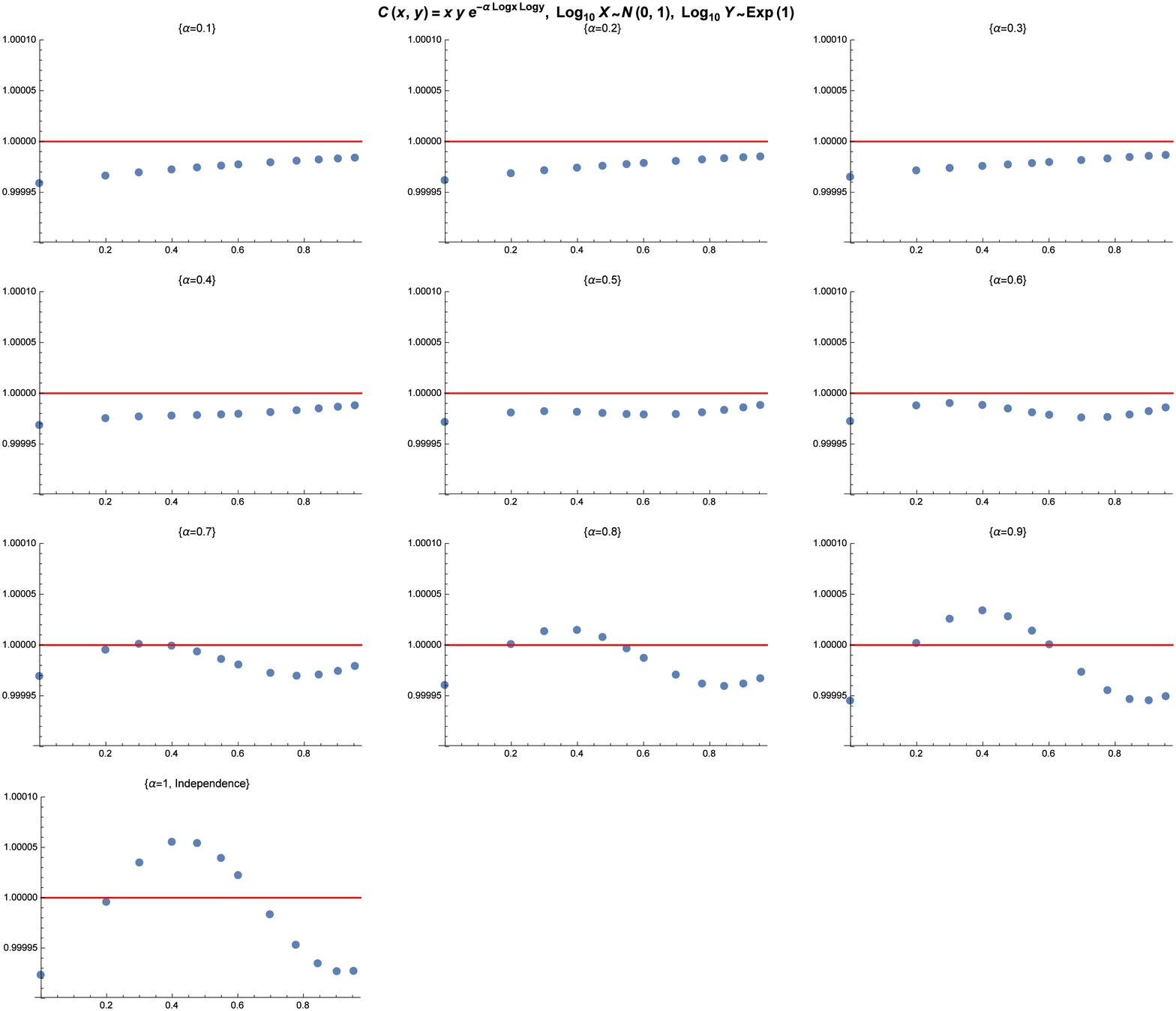}
\end{minipage}
\begin{minipage}[t]{0.32\linewidth}
\includegraphics[width=\linewidth,height=1.5\linewidth]{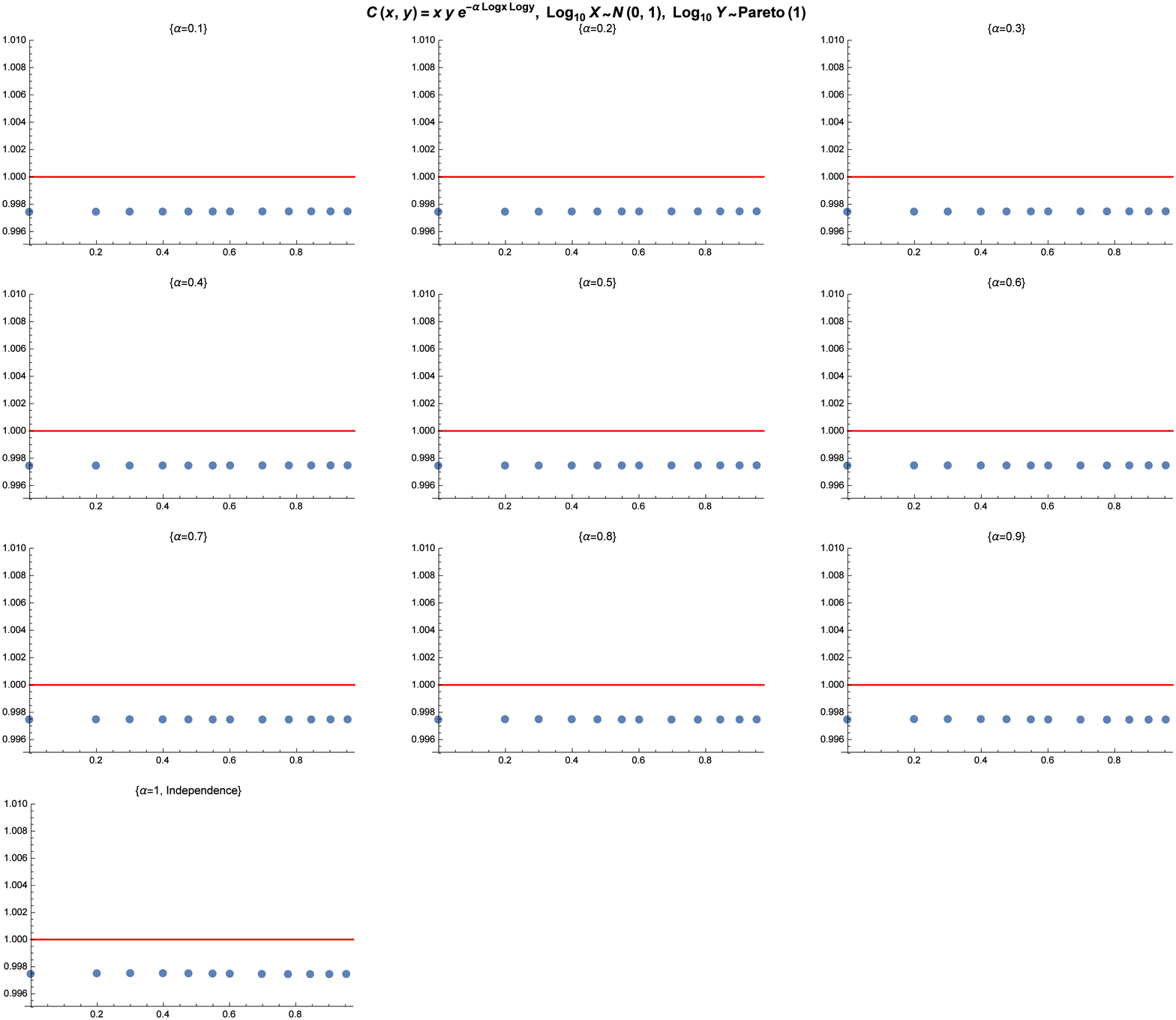}
\end{minipage}
\begin{minipage}[t]{0.32\linewidth}
\includegraphics[width=\linewidth,height=1.5\linewidth]{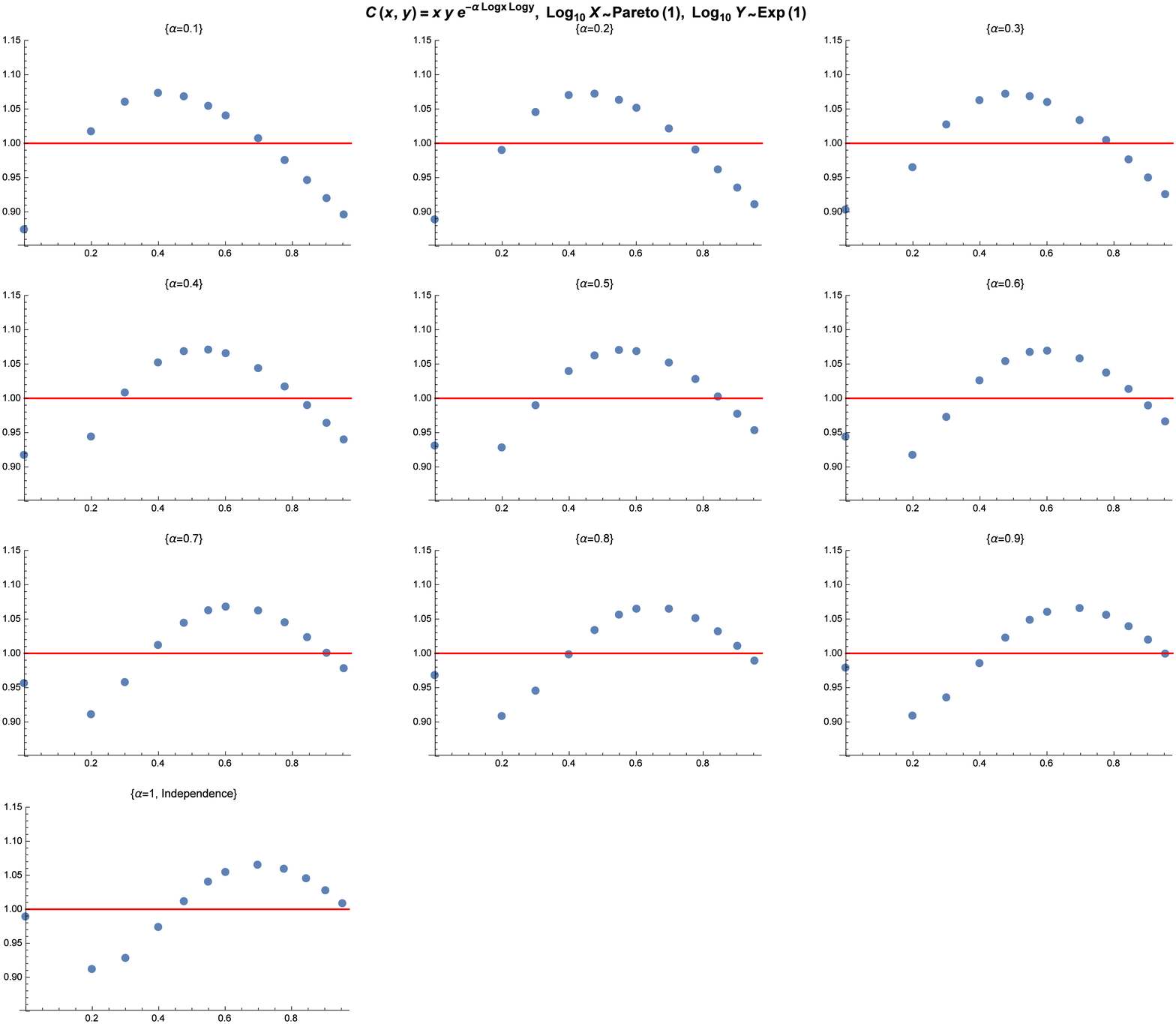}
\end{minipage}
\caption{The Gumbel-Barnett 2-copula (see Definition \ref{def:gb}) modeled on three different sets of marginals with varying dependence parameter $\alpha \in (0,1]$. The $y$-axes of these plots represent the approximate values of the copula PDF of $\log_{10}{XY} \bmod 1$ at various values of $x \in [0,1]$, where $X$ and $Y$ are the marginal distributions. The red line represents the Benford distribution.}\label{fig:gumho1}
\end{figure}
\FloatBarrier

\begin{figure}[h]
\includegraphics[width=\linewidth]{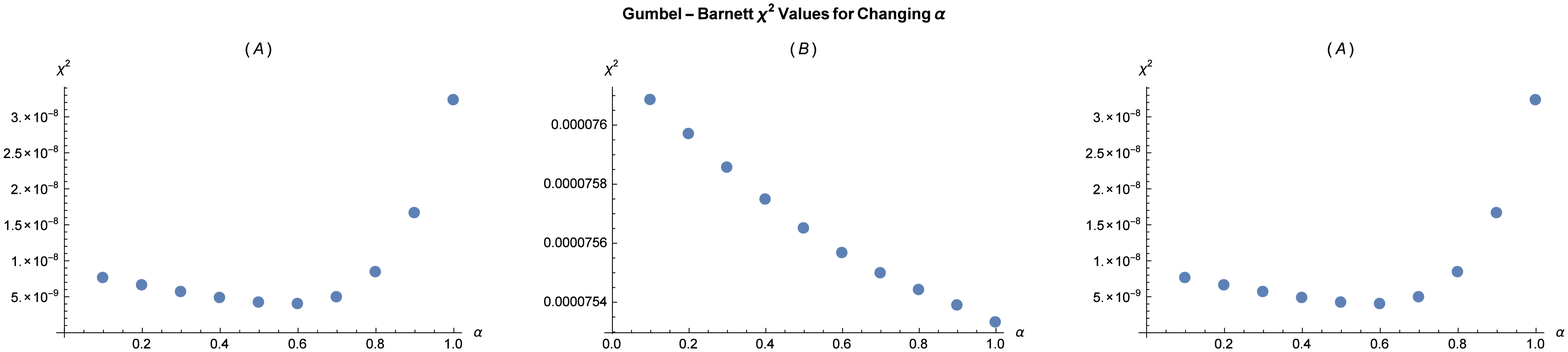}
\caption{The $\chi^{2}$ values associated to the the preceding sets of plots for the Gumbel-Barnett copula. Each shows the comparison to Benford behavior as $\alpha$ increases. We have $11$ degrees of freedom and a significance level of $0.005$, so we reject the hypothesis if the value exceeds $2.6$. Despite the apparent variation, none of these cases approach the critical value.} \label{fig:chi2}
\end{figure}
\FloatBarrier

\par
\subsubsection{Clayton Copula:}
Unlike the previous two examples, the Clayton copula shows notable variance over $\alpha$. Although it is not shown here, the independence case for Clayton copulas is $\alpha = 0$. For pairings $(A)$ and $(B)$, it appears that the plots diverge farther and farther away from $y=1$ as $\alpha$ moves away from $0$. For pairing $(C)$, the plots appear to get more random as $\alpha$ grows, and there is no suggestion that Benford behavior may develop as we depart from independence. Furthermore, the plots in Figure \ref{fig:chi3} show $\chi^{2}$ values that are significantly higher than those seen for the previous two copulas, suggesting that the dependence imposed by Clayton copula tends to heavily alter any Benford behavior of the marginals.
\FloatBarrier
\begin{figure}
\begin{minipage}[t]{0.32\linewidth}
\includegraphics[width=\linewidth,height=1.5\linewidth]{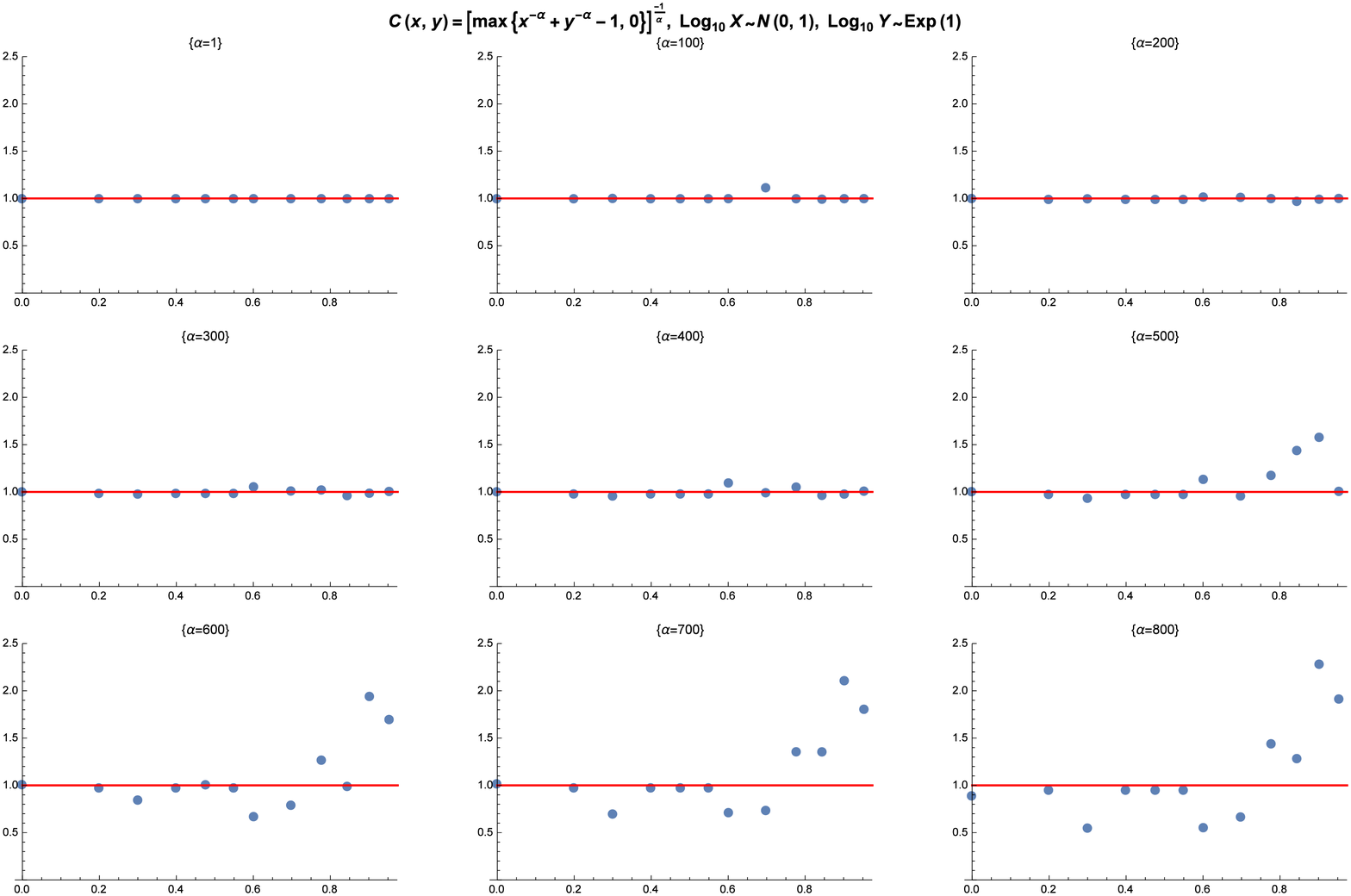}
\end{minipage}
\begin{minipage}[t]{0.32\linewidth}
\includegraphics[width=\linewidth,height=1.5\linewidth]{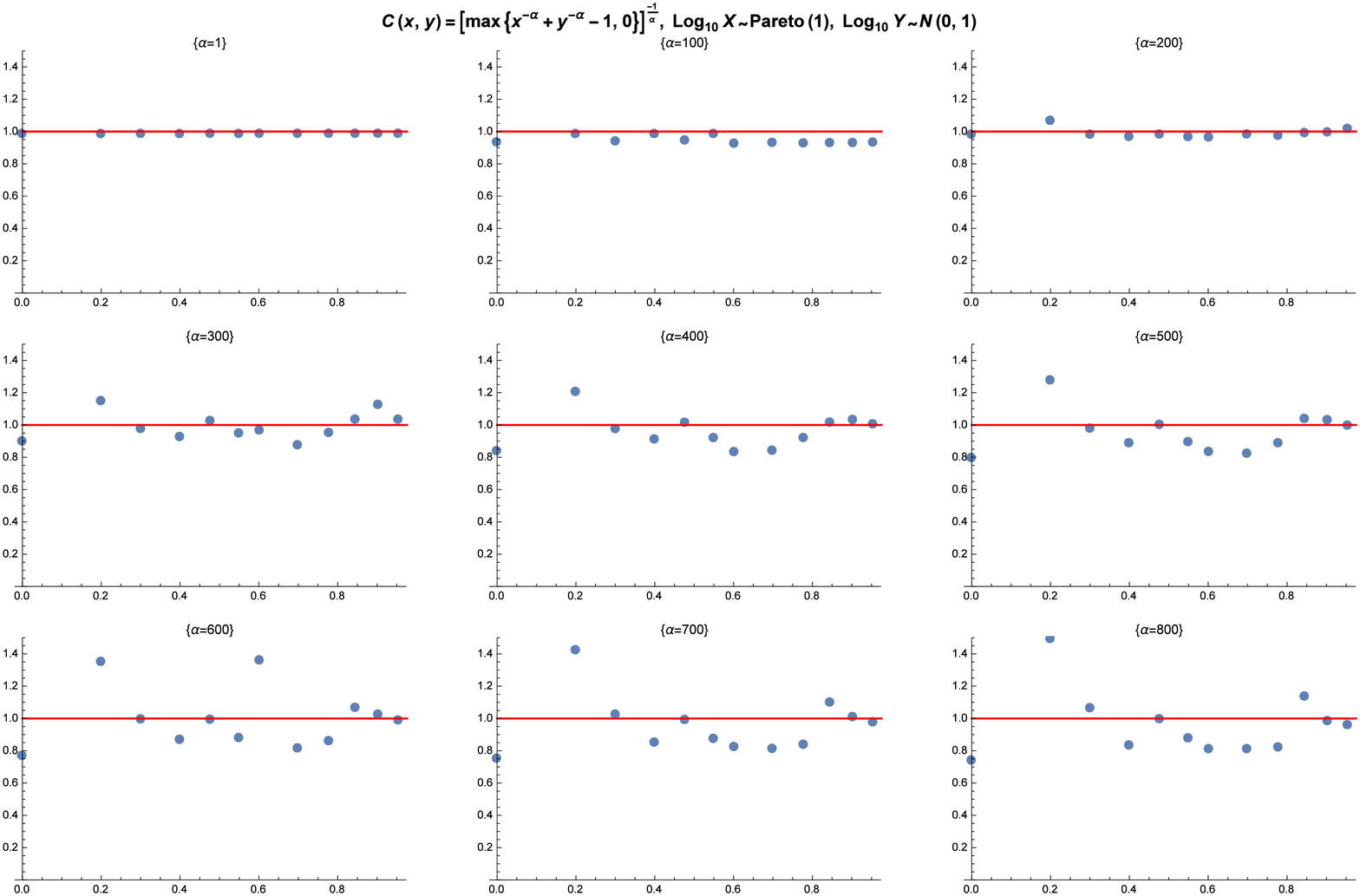}
\end{minipage}
\begin{minipage}[t]{0.32\linewidth}
\includegraphics[width=\linewidth,height=1.5\linewidth]{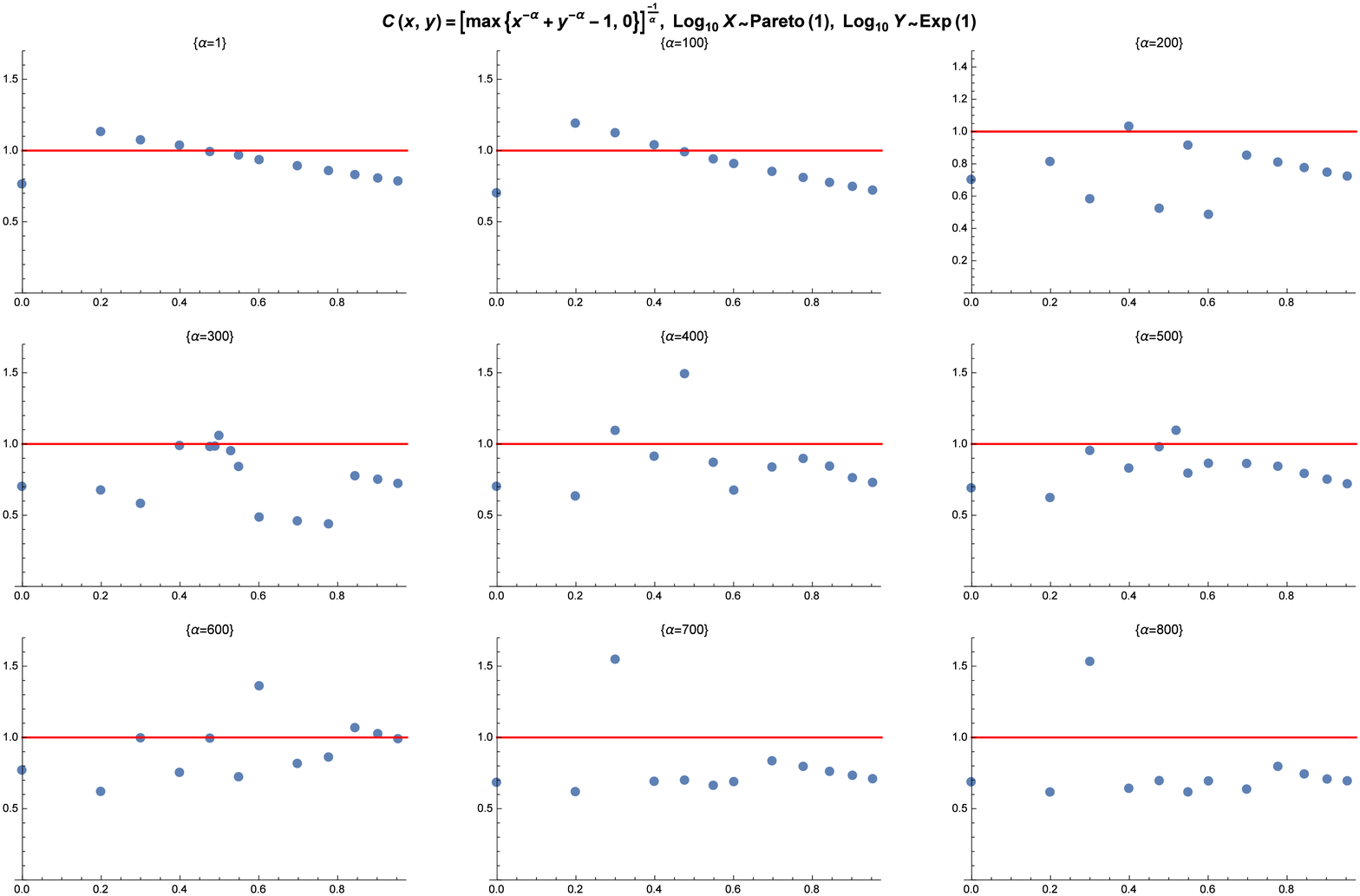}
\end{minipage}
\caption{The Gumbel-Barnett 2-copula (see Definition \ref{def:gb}) modeled on three different sets of marginals with varying dependence parameter $\alpha \in (0,1]$. The $y$-axes of these plots represent the approximate values of the copula PDF of $\log_{10}{XY} \bmod 1$ at various values of $x \in [0,1]$, where $X$ and $Y$ are the marginal distributions. The red line represents the Benford distribution.}\label{fig:clayton1}
\end{figure}
\FloatBarrier

\begin{figure}[h]
\includegraphics[width=\linewidth]{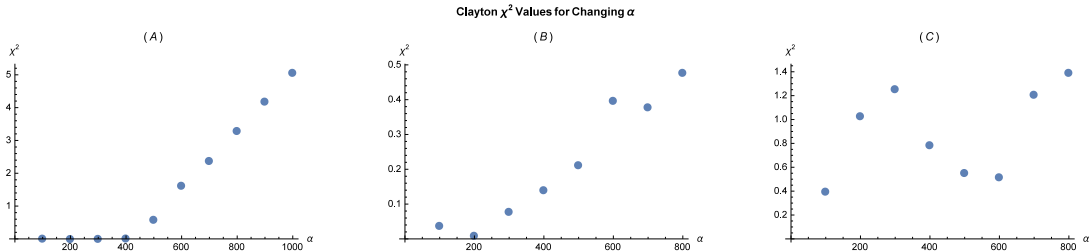}
\caption{The $\chi^{2}$ values associated to the the preceding sets of plots for the Clayton copula. Each shows the comparison to Benford behavior as $\alpha$ increases. We have $11$ degrees of freedom and a significance level of $0.005$, so we reject the hypothesis if the value exceeds $2.6$. Unlike the previous two copulas, only case $(B)$ stays below the critical value. However, the behavior of the plot suggests it will quickly surpass the critical value as $\alpha$ continues to increase.} \label{fig:chi3}
\end{figure}
\FloatBarrier

\par
The results from these three copulas suggest that the preservation of Benford behavior relies more heavily on the underlying structure of the copula than on the Benford behavior of the marginals. Both the Ali-Mikhail-Haq copula and the Gumbel-Barnett copula formulas contain the independence copula, $C(x,y)=xy$. The Clayton copula, however, does not contain the independence copula and is also the only copula of the three to show noticeable variation as the dependence parameter changes.
\subsection{$n$-Copulas}
The previous results suggest that the underlying copula structure has a strong influence on the Benford behavior of $2$-copulas. Thus the logical next step is to investigate whether this holds true as we increase the number of Marginals. For all $\chi^{2}$ tests, we have $8$ degrees of freedom and again take a significance level of $0.005$. In practice, this means we reject the hypothesis if the value exceeds $1.3$.
\par
We consider the most stable of the three previous copulas, the Gumbel-Barnett copula. We fix $\alpha = 0.1$ and set the log, base 10, of all marginals to be identically distributed according to the Normal distribution with mean $0$ and variance $1$, our most Benford-like marginal. We then consider cases where the copula has $2$ to $7$ marginals. We can see from Figure (\ref{fig:1257}) that the Benford behavior of the Gumbel-Barnett copula begins to fall apart as marginals are added. This is in direct contrast to what would be expected from a central-limit type property, which should become increasingly more uniform as variables are added. This is further reinforced by the $\chi^{2}$ values in Figure \ref{fig:chimulti} and suggests that the dependence structure imposed by the copula prevents any leveling-off from happening.
\begin{figure}[h]
\includegraphics[ width=0.5\textwidth,height=0.5\textwidth]{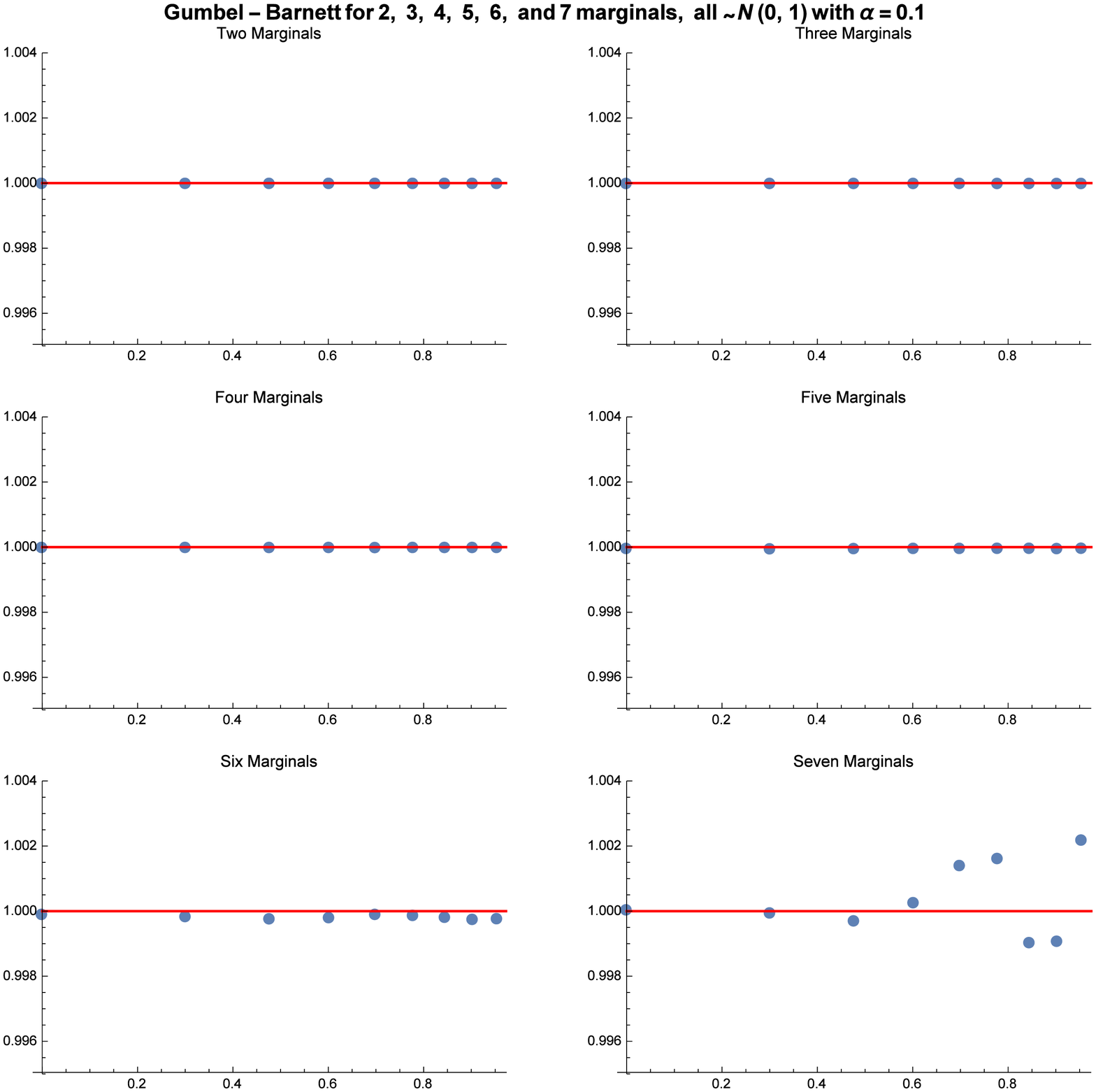}
\caption{Gumbel-Barnett copula with two to seven marginals}\label{fig:1257}
\end{figure}
\FloatBarrier

\begin{figure}[h]
\includegraphics[width=0.4\linewidth]{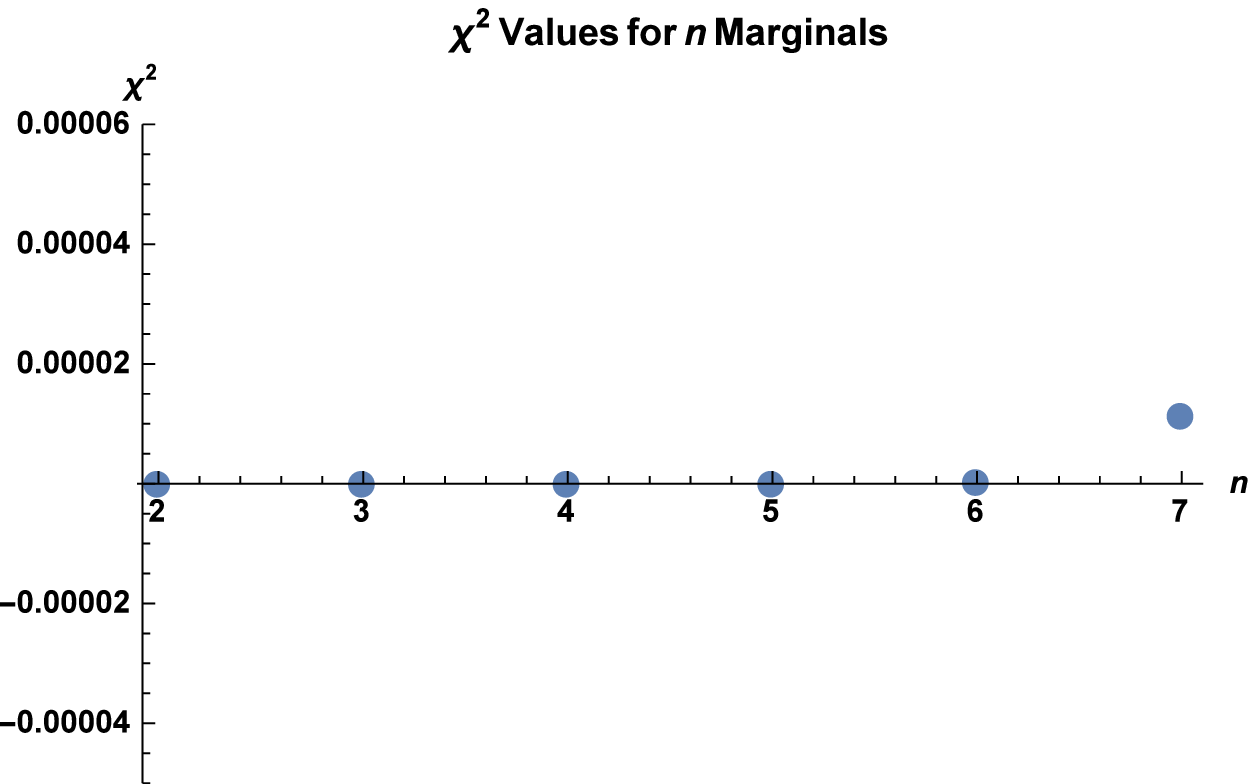}
\caption{The $\chi^{2}$ values comparing the behavior of the product to a Benford PDF as the number of marginals increases.  We have $8$ degrees of freedom and a significance level of $0.005$, so we reject the hypothesis if the value exceeds $1.3$.} \label{fig:chimulti}
\end{figure}
\FloatBarrier

\section{Benford Distance}\label{sec: sectionfour}
Now that we know that we can test for Benford behavior of a product, regardless of dependence, it would be prudent to know how often this behavior is expected to show up. In order to do this, we investigate if the absorptive property of Benford products is common in dependent random variables, or if its presence relies on some sort of proximity to independence.

To get an idea of this, let $\mathcal{W}$ be the space of all $n$-tuples of continuous random variables $(X_{1},X_{2},\dots,X_{n})$ for which at least one is Benford. Now let us assume that our set of marginals, $(X_{1},X_{2},\dots,X_{n})$, form an element in $\mathcal{W}$. Then we know that their product, assuming independence, will always be Benford.

From this, we can restrict our Benford distance, (\ref{eq:eps}), to $\mathcal{W}$ and define it as
\begin{eqnarray}\label{eq:epsW}
& & \epsilon_{s,W} \ = \ \nonumber\\
& & \left| \sum_{k=-\infty}^{\infty}\int_{u_{1}=-\infty}^{\infty}\cdots \int_{u_{n-1}=-\infty}^{\infty}\left(1 \ - \ \frac{\partial^{n}C(F_{1}(u_{1}),\dots,F_{n-1}(u_{n-1}),F_{n}(u_{n}))}{\partial{u_1}\partial{u_2}\cdots \partial{u_n}}\Big|_{\textbf{u}_0}du_{1}du_{2}\cdots du_{n-1}\right) \  \right|, \nonumber\\
\end{eqnarray}
where $\textbf{u}_0$ is defined as in Lemma \ref{thm: benfsolve}.
Therefore, our problem becomes to minimize the value of $\epsilon_{s,W}=0,$ as a proximity to $0$ should indicate proximity to a Benford distribution.

\subsection{Cases that are $\epsilon$ away from Benford.}

Rather than directly calculating the value of $\epsilon_{s,W}$, it may often be more convenient to provide a bound that depends only on the copula $C$. Note that if the value of $\frac{\partial^{n}C(F_{1}(u_{1}),\dots,F_{n-1}(u_{n-1}),F_{n}(u_{n}))}{\partial{u_1}tial{u_2}\cdots \partial{u_n}}$ is identically $1$ for all values of $(u_{1},u_{2},\dots,u_{n})$, then the value of $\epsilon_{s,W}$ will be identically $0$ and our product will be Benford. Even though this case does not cover all situations in which our product will be Benford, it suggests that a product's distance from Benford may be related to the distance between the function $\frac{\partial^{n}C(F_{1}(u_{1}),\dots,F_{n-1}(u_{n-1}),F_{n}(u_{n}))}{\partial{u_1}\partial{u_2}\cdots \partial{u_n}}$ and the constant function, $1$. This brings us to the main result of this section.

\begin{thm}\label{thm:L1normBounds}
Let $X_{1},X_{2},\dots,X_{n}$ be continuous random variables where $(X_{1},X_{2},\dots,X_{n}) \in \mathcal{W}$. Assume also that they are jointly described by a copula $C$, where the function $N(u_{1},u_{2},\dots,u_{n}) = 1-\frac{\partial^{n}C(F_{1}(u_{1}),\dots,F_{n-1}(u_{n-1}),F_{n}(u_{n}))}{\partial{u_1}\partial{u_2}\cdots \partial{u_n}}$ is in $L^{1}(\mathbb{R}^{n})$. Let $U_{i}=\log_{B}{X_{i}}$ for each $i$ and some base, $B$, and let $F_{i}$ be the CDFs of $U_{i}$ for each $i$. Then the \textit{$L^{1}$ distance from Benford}, defined by
\begin{equation}\label{eq:accDist}
\int_{0}^{1} \left| \sum_{k=-\infty}^{\infty}\int_{u_{1}=-\infty}^{\infty}\cdots \int_{u_{n-1}=-\infty}^{\infty} \left( 1 \ - \ \frac{\partial^{n}C(F_{1}(u_{1}),\dots,F_{n-1}(u_{n-1}),F_{n}(u_{n}))}{\partial{u_1}tial{u_2}\cdots \partial{u_n}}\Big|_{\textbf{u}_{0}}\right)du_{1}\cdots du_{n-1} \ \right| ds
\end{equation}
is bounded above by the $L^{1}$ norm of $N$. In other words
\begin{align}\label{eq: L1Bound}
\nonumber & \int_{0}^{1} \left| \sum_{k=-\infty}^{\infty}\int_{u_{1}=-\infty}^{\infty}\cdots \int_{u_{n-1}=-\infty}^{\infty} \left( 1 \ - \ \frac{\partial^{n}C(F_{1}(u_{1}),\dots,F_{n-1}(u_{n-1}),F_{n}(u_{n}))}{\partial{u_1}tial{u_2}\cdots \partial{u_n}}\Big|_{\textbf{u}_{0}}\right)du_{1}\cdots du_{n-1} \ \right| ds \\
& \ \leq \ \|1 \ - \ \frac{\partial^{n}C(F_{1}(u_{1}),\dots,F_{n-1}(u_{n-1}),F_{n}(u_{n}))}{\partial{u_1}\partial{u_2}\cdots \partial{u_n}}\|_{L^{1}}.
\end{align}
\end{thm}

We prove this for the two-dimensional case, as the results in $n$-dimensions proceed similarly. We need the following result (see Appendix \ref{sec:Appendix 1} for a proof).

\begin{lem}\label{lem:L1Lemma}
Given $C_{uv}$, $F(u)$, and $G(v)$ as defined before, we have
\begin{equation}
\| 1 - C_{uv}(u,v) \|_{L^{1}} \ = \ \int_{-\infty}^{\infty}\int_{-\infty}^{\infty}f(u)g(v)|1-C_{uv}(F(u),G(v))|dudv.
\end{equation}
\end{lem}


\begin{proof}[Proof of Theorem \ref{thm:L1normBounds}]
From the positivity of $f$ and $g$ we have
\begin{align}\label{eq: L1start}
\nonumber&\int_{0}^{1} \left| \sum_{k=-\infty}^{\infty}\int_{-\infty}^{\infty}f(u)g(s+k-u)(1 \ - \ C_{uv}(F(u),G(s+k-u)))du \  \right|ds \\
& \ \leq \ \int_{0}^{1}  \sum_{k=-\infty}^{\infty}\int_{-\infty}^{\infty}f(u)g(s+k-u)|1 \ - \ C_{uv}(F(u),G(s+k-u))|du \  ds.
\end{align}

We investigate exactly what region \eqref{eq: L1start} covers. The lines shown in Figure \ref{fig:region} are the sets $A_{k} = \{(u,v) : v=s+k-u\}$. We integrate $f(u)g(s+k-u)(1-C_{uv}(F(u),G(s+k-u)))$ along each of these lines and sum the results over $k$. The shaded region shows the area covered when $A_{2}$ is integrated over $s$ from $0$ to $1$.

\begin{figure}[h]
\includegraphics[width=0.5\textwidth, height=0.5\textwidth]{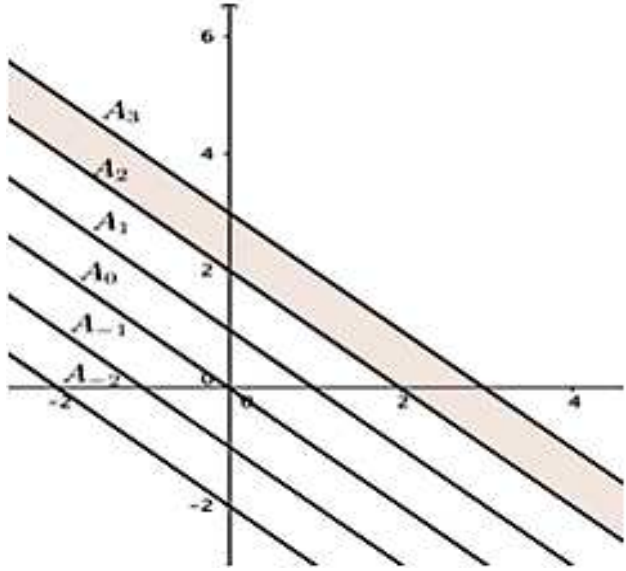}
\caption{The plane broken up into a few of the sections $A_{k}$.}\label{fig:region}
\end{figure}

As all of our sums and integrals converge absolutely, by Fubini's theorem we may switch our sum and integral in \eqref{eq: L1start} and get
\begin{align}\label{eq:switch}
\nonumber &\int_{0}^{1}  \sum_{k=-\infty}^{\infty}\int_{-\infty}^{\infty}f(u)g(s+k-u)|1 \ - \ C_{uv}(F(u),G(s+k-u))|du \  ds \\
& \ = \  \sum_{k=-\infty}^{\infty}\int_{0}^{1} \int_{-\infty}^{\infty}f(u)g(s+k-u)|1 \ - \ C_{uv}(F(u),G(s+k-u))|du \  ds.
\end{align}
From this, we can quickly see that for any $k$,
\begin{equation}\label{eq:stillL1Proof}
\int_{0}^{1} \int_{-\infty}^{\infty}f(u)g(s+k-u)|1 \ - \ C_{uv}(F(u),G(s+k-u))|du \  ds
\end{equation}
is the integral of $f(u)g(s+k-u)|1 \ - \ C_{uv}(F(u),G(s+k-u))|$ over a region in between and including $A_{k}$ and $A_{k+1}$, just like the shaded region in Figure \ref{fig:region}. Therefore, (\ref{eq:switch}) is the sum of the integrals of $f(u)g(s+k-u)|1 \ - \ C_{uv}(F(u),G(s+k-u))|$ over all of these (disjoint) regions (over all $k$), which is equivalent to integrating over all of $\mathbb{R}^{2}$, giving us
\begin{equation}
\int_{-\infty}^{\infty}\int_{-\infty}^{\infty}f(u)g(v)|1 \ - \ C_{uv}(F(u),G(v))|dudv.
\end{equation}
Finally, we know from Lemma \ref{lem:L1Lemma}, we know that this is equal to $\|1-C_{uv}(u,v)\|_{L^{1}}$.
\end{proof}

\subsection{Consequences of an $L^{1}$ bound in $\mathbb{R}^{2}$.}

What Theorem \ref{thm:L1normBounds} provides is a way to understand the behavior of our probabilities. To see this, let $\mathcal{S} \subset [0,1]$ be the region over which $\epsilon_{s,W}>\epsilon_{N}$. If $\epsilon_{s,W}$ is large on $\mathcal{S}$, then the measure of $\mathcal{S}$ must be small in order to conform to \eqref{eq: L1Bound}, which requires that if $\|1-C_{uv}(u,v)\|_{L^{1}} \leq \epsilon_{N}$, then $\int_{0}^{1}\epsilon_{s,W}ds \leq \epsilon_{N}$ as well. In fact, the following corollary proves that Theorem \ref{thm:L1normBounds} provides useful information regarding how large $|\mathcal{S}|$ can be.

\begin{cor}\label{thm:corollary}
Let $\mathcal{S} \subset [0,1]$ be the set $\{s : \epsilon_{s,W} \geq \epsilon \}$. Then
\begin{equation}\label{eq:fromMark}
|\mathcal{S}| \ \leq \ \frac{\|1-C_{uv}(u,v)\|_{L^{1}} }{\epsilon}.
\end{equation}
\end{cor}

\begin{proof}
This result comes directly from Markov's Inequality:
\begin{equation}\label{eq:usingMARK}
|\{s : \epsilon_{s,W} \geq \epsilon \}| \ \leq\ \frac{1}{\epsilon}\int_{0}^{1}\epsilon_{s,W} \ \leq \ \frac{\|1-C_{uv}(u,v)\|_{L^{1}} }{\epsilon}.
\end{equation}
\end{proof}



\section{Applications, Future Work, and Conclusion}\label{sec: sectionfive}

\subsection{Fitting Copulas}

The results of Section \ref{sec: sectionthree} allow us to determine the Benford behavior of the product $n$ distributions jointly modeled by a specific copula. However, we may wish to go in the other direction and, instead, find a copula that best fits $n$ correlated data sets. Statisticians have several methods for testing the goodness-of-fit to find the best choice of copula in these situations (see \cite{Genest3} for some examples and an analysis of several forms of goodness-of-fit tests), but it is not known whether or not these goodness-of-fit tests take Benford behavior into account. That is to say, will the prescribed copula mimic the Benford behavior observed in the data?

\par
The results Section \ref{sec: sectionthreepointfive} have shown us that the product of the same set of marginals will not display the same Benford behavior when modeled by different copulas. Thus, Benford behavior is not guaranteed. A natural next step is to investigate how the goodness-of-fit of a copula  may or may not be correlated with how well it preserves the expected Benford behavior of the product of two or more marginals. A comparison between the $L^{1}$ norm and well-known goodness of fit tests would enable us to see whether or not a strong Benford fit corresponds to a well-fit distribution as a whole. Furthermore, if a stronger Benford fit may be shown to correspond to a smaller $L^1$ bound, then we may be able to define this bound as a new goodness of fit test for distributions with one or more Benford marginals.



\subsection{Conclusion}

In fields such as actuarial sciences and statistics Benford's law is useful for fraud detection. Furthermore, copulas are a highly effective tool for modelling systems with dependencies.  In Section \ref{sec: sectionthree} we demonstrated that Benford behavior for dependent variables modeled by a copula may be detected and therefore analyzed to investigate the product of the variables. Thus these results indicate that the Benford's law methods used by professionals on single-variate, and/or independent data sets are now at the disposal of individuals who wish to model dependent data via a copula. We then applied these results in Section \ref{sec: sectionthreepointfive} where we observed that the preservation of Benford behavior appears to rely more heavily on the structure of the copula than on the marginals.

Essentially, the results of Section \ref{sec: sectionthree} permit analyses like those carried out in \cite{Cuff} and \cite{D--} in which a known distribution, in these cases the Weibull distribution and the inverse-gamma distribution, is analyzed to determine the conditions under which Benford behavior should arise. Once these conditions are established, any non-Benford data set which is expected to come from such a distribution may be considered suspicious enough to warrant a fraud investigation. In the case of copulas, the results of Section \ref{sec: sectionthree} allow one to conduct this exact method of analysis on the product of $n$ random variables jointly modeled by a copula $C$.

Finally, in Section \ref{sec: sectionfour} we encountered a useful consequence of of considering a distribution's $L^{1}$ distance from a Benford distribution to determine a useful bound for this Benford distance. We determined that the Benford distance of a product of $n$ random variables will always be bounded above by the distance between the copula PDF and the class of copulas whose PDFs are identically $1$.


\appendix

\section{Proofs for supporting Lemmas and Theorems}\label{sec:Appendix 1}

\textbf{Proof of Lemma \ref{thm: benfsolve}.}

\textit{Given $X$ and $Y$ continuous random variables with joint distribution modeled by the absolutely continuous copula $C$, Let $U=\log_{B}{X}$ and $V=\log_{B}{Y}$, for some base, $B$, and let the (marginal) CDFs of $U$ and $V$ be $F(u)$ and $G(v)$, respectively. Also, let $f(u)$ and $g(v)$ be the PDFs of $U$ and $V$, respectively. Then}
\begin{align}
\nonumber&\Prob{(U+V)\mod{1} \leq s} \\
& = \ \int_{0}^{s} \left( \sum_{k=-\infty}^{\infty}\int_{-\infty}^{\infty}C_{uv}(F(u),G(s+k-u))f(u)g(s+k-u)du \right).
\end{align}
\textit{Therefore, the PDF of $(U+V)\mod{1}$ is given by}
\begin{equation}
\sum_{k=-\infty}^{\infty}\int_{-\infty}^{\infty}C_{uv}(F(u),G(s+k-u))f(u)g(s+k-u)du.
\end{equation}

\begin{proof}
By the invariance of copulas under monotonically increasing functions (Theorem \ref{thm:noChange}), we know that the joint CDF of $U$ and $V$ is given by the same copula as $X$ and $Y$. Thus, the joint CDF of $U$ and $V$ is given by
\begin{equation}
C(F(U),G(V)).
\end{equation}

Then, by definition, the joint PDF of $U$ and $V$ is given by the mixed partial derivative.
\begin{align}\label{eq:pdfUV}
\frac{\partial}{\partial{v}}\frac{\partial}{\partial{u}}C(F(u),G(v)) & \ = \ C_{uv}(F(u),G(v))f(u)g(v) \ + \ C_{u}(F(u),G(v))\frac{\partial}{\partial{v}}f(u) \nonumber\\
& \ = \ C_{uv}(F(u),G(v))f(u)g(v).
\end{align}
Note that we assume that $\frac{du}{dv}=0$ since all dependence between $U$ and $V$ is modeled by $C$.

Note, also, that $\Prob{XY \leq 10^{s}}=\Prob{(U+V) \leq s}$. Thus we have
\begin{align} \label{eq: cdfmod1}
\nonumber&\Prob{(U+V)\mod{1} \leq s} \\
& = \ \sum_{k=-\infty}^{\infty}\int_{u=-\infty}^{\infty}\int_{v=k-u}^{s+k-u}C_{uv}(F(u),G(v))f(u)g(v)dvdu.
\end{align}

If $XY$ is Benford, then (\ref{eq: cdfmod1}) will equal $s$ for all $s$. It is, however, easier to test the PDF then the CDF. So we differentiate with respect to $s$. Let $C_{1}(u,v)$ be the antiderivative of $C_{uv}(F(u),G(v))f(u)g(v)$ with respect to $v$. Then
\begin{align}
\nonumber& \frac{\partial}{\partial{s}}\sum_{k=-\infty}^{\infty}\int_{u=-\infty}^{\infty}\int_{v=k-u}^{s+k-u}C_{uv}(F(u),G(v))f(u)g(v)dvdu \\
\nonumber&  \ = \ \frac{\partial}{\partial{s}}\sum_{k=-\infty}^{\infty}(\int_{u=-\infty}^{\infty}(C_{1}(u,s+k-u)-C_{1}(u,k-u))du \\
& \ = \ \sum_{k=-\infty}^{\infty}\int_{-\infty}^{\infty}C_{uv}(F(u),G(s+k-u))f(u)g(s+k-u)du.
\end{align}
\end{proof}

\textbf{Proof of Lemma \ref{lem:convergence}.}

\textit{Given $U$ and $V$, continuous random variables modeled by the copula $C$ with marginals $F$ and $G$, respectively, there exist $a_{1}$, $a_{2}$, $b_{1}$, and $b_{2}$  completely dependent on $F$ or $G$ such that $a_{1}< a_{2}$ and $b_{1}<b_{2}$, and}
\begin{align}
\nonumber &\sum_{k=-\infty}^{\infty}\int_{-\infty}^{\infty}C_{uv}(F(u),G(s+k-u))f(u)g(s+k-u)du \\
&\ = \ \sum_{k=b_{1}}^{b_{2}}\int_{a_{1}}^{a_{2}}C_{uv}(F(u),G(s+k-u))f(u)g(s+k-u)du \ + \ E_{a,b}(s)
\end{align}
\textit{where $E_{a,b}(s) \to 0$ as $a_{1},b_{1} \to -\infty$ and $a_{2},b_{2} \to \infty$. Thus, for any $\epsilon > 0$, there exists $|a_{1}|$, $|a_{2}|$, $|b_{1}|$, and $|b_{2}|$ large enough such that $|E_{a,b}(s)| \leq \epsilon$.}

\begin{proof}
Since both the sum and the integral are convergent, the proof for $a_{1}$, $a_{2}$ and $b_{1}$, $b_{2}$ are nearly identical, so we only provide the work here for $a_{1}$ and $a_{2}$. The same steps may be used in the proof for $b_{1}$ and $b_{2}$. We also know that $C_{uv}(F(u),G(s+k-u))f(u)g(s+k-u)$ must go to $0$ as $u$ goes to $\pm \infty$ because of this convergence. Thus we choose to prove the case where $f$ and/or $g$ converge faster than $C_{uv}$. If $C_{uv}$ were to converge faster, the results derived here would still suffice.
We prove that for any $\epsilon >0$ we can find $a_{1}$ and $a_{2}$ such that, for all $u \leq a_{1}$ and all $u \geq a_{2}$, we have
\begin{equation*}
|C_{uv}(F(u),G(s+k-u))f(u)g(s+k-u)| \ \leq \ \epsilon.
\end{equation*}

Let $\epsilon > 0$, set $s$ and $k$ to be constant, and assume $C_{uv}$ is nonzero everywhere. If $C_{uv}$ is zero at any point, then we have a trivial case. Because $F$ and $G$ are CDFs, we know that $f \to 0$ as $u \to \pm \infty$ and $g\to 0$ as $ -u \to \pm \infty$, thus, we may choose $a_{f1}$, $a_{f2}$, $a_{g1}$, and $a_{g2}$ such that, for all $u \leq a_{f{1}}$ and all $u \geq a_{f2}$, we have
\begin{equation}\label{eq:fcontrol}
f(u) \ \leq\ \sqrt{\frac{\epsilon}{C_{uv}(F(u)G(s+k-u))}}.
\end{equation}
The same can be done for $g$ such that, for all $u \geq a_{g1}$ and all $u \leq a_{g2}$, we have
\begin{equation}
g(s+k-u) \ \leq\  \sqrt{\frac{\epsilon}{C_{uv}(F(u)G(s+k-u))}}.
\end{equation}
Thus, we let $a_{1} = \min\{a_{f1},a_{g1}\}$ and $a_{2} = \max\{a_{f2},a_{g2}\}$. then we have, for all $u \leq a_{1}$ and all $u \geq a_{2}$, we have
\begin{equation*}
|C_{uv}(F(u),G(s+k-u))f(u)g(s+k-u)| \ \leq \ \epsilon.
\end{equation*}
\end{proof}

\textbf{Proof of Lemma \ref{lem:L1Lemma}.}

\textit{Given $C_{uv}$, $F(u)$, and $G(v)$ as defined in Theorem \ref{thm:L1normBounds}, we have}
\begin{equation}
\| 1 - C_{uv}(u,v) \|_{L^{1}} \ = \ \int_{-\infty}^{\infty}\int_{-\infty}^{\infty}f(u)g(v)|1-C_{uv}(F(u),G(v))|dudv.
\end{equation}
\begin{proof}
We know that $u$ and $v$ are defined on $[0,1]$. Thus,
\begin{equation}
\| 1 - C_{uv}(u,v) \|_{L^{1}} \ = \ \ \int_{0}^{1}\int_{0}^{1}|1-C_{uv}(u,v)|dudv.
\end{equation}
However, by a simple change of variables $u \to F(u)$, $v \to G(v)$ (defended as CDFs, just like before, so their derivatives are $f(u)$ and $g(v)$, both of which are greater than or equal to 0), we get
\begin{equation}
\| 1 - C_{uv}(u,v) \|_{L^{1}} \ = \ \int_{-\infty}^{\infty}\int_{-\infty}^{\infty}f(u)g(v)|1-C_{uv}(F(u),G(v))|dudv.
\end{equation}
\end{proof}


\section{Computationally Testing for Benford Behavior: Examples}\label{sec:Appendix 2}
In this section, we use Clayton copulas (see Definition \ref{def:clayton}) to determine the Benford behavior of different combinations of marginals. We specifically look at marginals of the form $X=10^{U}$ and $Y=10^{V}$, where $U$ and $V$ are $\rm N[0,1]$ or $\rm Exp[1]$. In all analyses, we let $\alpha=2$ and $B=10$. We also provide the independence case for each set of marginals to allow for comparison.

\ \\

\noindent\textbf{\textit{Case 1: $U$ and $V$ $\sim N[0,1]$.}}

Given our definition of $X$ and $Y$, (\ref{eq:pdfmod1}) we first determine acceptable values for $a_{1}$, $b_{1}$, $a_{2}$. and $b_{2}$ by using an error analysis to test whether or not $-10$ and $10$ should be acceptable values for $a_{1}$ and $a_{2}$.

We generated a list of the error caused by truncating the integral at these values for various values of $s$. The first value of each triple in the list is $s$. The second is the lower error and the third is the upper error. To determine the error caused by truncating the integral, we used the approximation method detailed in Section \ref{sec: sectionthree}. As the list shows, the error is on the order of $10^{-22}$ or smaller, indicating that our selections for $a_{1}$ and $a_{2}$ are good bounds. We took the sum from $k=-20$ to $k=20$ because we know this will be sufficient, as indicated by the convergence in Figure (\ref{fig: plot2}) below.

\begin{verbatim}
In[262]:= errorsb =
 Table[{N[Log[10, s]], ea[Log[10, s]], eb[Log[10, s]]}, {s, 1, 9}]

Out[262]= {{0., 6.86784*10^-22, 1.28213*10^-22}, {0.30103,
  9.38169*10^-24, 1.28257*10^-22}, {0.477121, 3.03058*10^-25,
  1.28274*10^-22}, {0.60206, 2.74232*10^-26,
  1.28266*10^-22}, {0.69897, 4.3443*10^-27,
  1.28249*10^-22}, {0.778151, 9.77379*10^-28,
  1.28234*10^-22}, {0.845098, 2.79567*10^-28,
  1.28223*10^-22}, {0.90309, 9.52164*10^-29,
  1.28216*10^-22}, {0.954243, 3.70245*10^-29, 1.28213*10^-22}}
  \end{verbatim}

We now plot in Figure (\ref{fig: plot2}) the value of our truncated form of our PDF for different values of $s$. The line $y=1$ is included to demonstrate how close to $1$ our PDF is for all values of $s$, suggesting that the product of $X$ and $Y$, with joint PDF modeled by a Clayton copula with $\alpha=2$ should display Benford behavior.

\begin{figure}[h]
\includegraphics[width=0.6\textwidth, height= 0.43\textwidth]{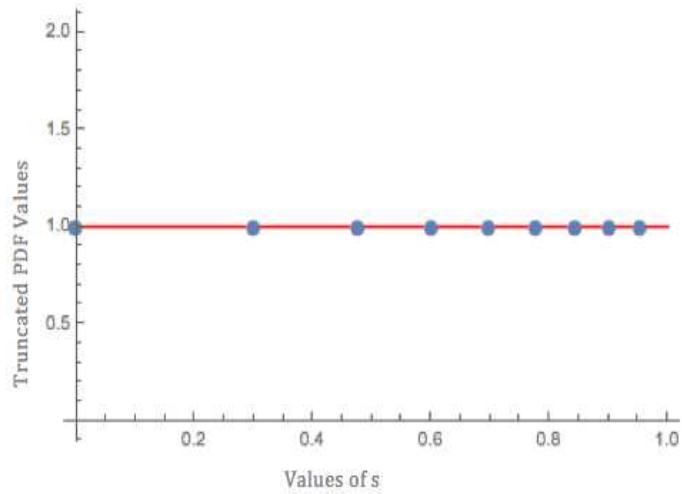}
\caption{$U \sim \rm{N[0,1]}$, $V \sim \rm{N[0,1]}$}\label{fig: plot2}
\end{figure}
\FloatBarrier

\ \\

\noindent\textbf{\textit{Case 2: $U$ $ \sim N[0,1]$ and $V$ $ \sim Exp[1]$.}}

A similar analysis as before was conducted on this new set of variables.  Through an identical analysis, we defined the bounds for our integral to be $a=-5$ and $b=10$, and provide the accumulated errors in the code below where the first term in each pair is $s$ and the second and third are the lower and upper errors, respectively.
As we can see, the errors are still very very small.

\begin{verbatim}
In[419]:= Table[{N[Log[10, s]], ea2[Log[10, s]], eb2[Log[10, s]]}, 
    {s, 1, 9}]

Out[419]= {{0., 3.30411*10^-21, 1.23628*10^-22}, {0.30103,
  2.43577*10^-21, 1.27151*10^-22}, {0.477121, 2.03887*10^-21,
  1.31758*10^-22}, {0.60206, 1.79746*10^-21,
  1.32924*10^-22}, {0.69897, 1.63021*10^-21,
  1.32387*10^-22}, {0.778151, 1.50526*10^-21,
  1.31045*10^-22}, {0.845098, 1.40717*10^-21,
  1.2933*10^-22}, {0.90309, 1.32741*10^-21,
  1.27456*10^-22}, {0.954243, 1.26084*10^-21, 1.25536*10^-22}}
\end{verbatim}

We now plot in Figure (\ref{fig: plot3}) the value of our truncated form of our PDF for various $s$.  We again note how close the PDF remains to $1$ for all values of $s$, suggesting that the product of $X$ and $Y$, with joint PDF modeled by a Clayton copula with $\alpha=2$ should display Benford behavior.

\begin{figure}[h]
\includegraphics[width=0.6\textwidth, height= 0.43\textwidth]{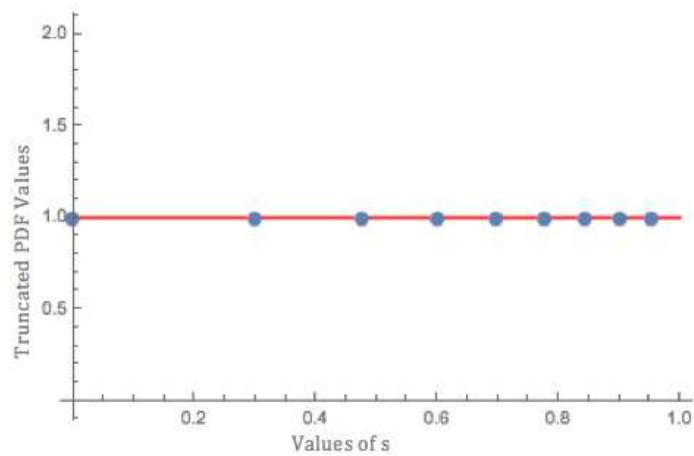}
\caption{$U \sim N[0,1],$ $V \sim Exp[1]$}\label{fig: plot3}
\end{figure}
\FloatBarrier

As before, this is backed up by the following simulation.
\\
\textbf{Simulation 2}

\ \\

\noindent\textbf{\textit{Case 3: $U$ $ \sim Exp[,1]$ and $V$ $ \sim Exp[1]$.}}

Finally, we conduct our analysis on the case of two exponentials. Our error terms for $a=25$ are generated in the code below (By inspection, we can tell that $C_{uv}(F(u),G(s+k-u))f(u)g(s+k-u)$ will be zero for negative values of $u$).  Again we choose $k$ from $0$ to $50$, and the first term in each pair is $s$.

\begin{verbatim}
In[363]:= Table[{N[Log[10, s]], N[eb1[Log[10, s]]]}, {s, 1, 9}]

Out[363]= {{0., 5.57839*10^-11}, {0.30103, 5.73736*10^-11}, {0.477121,
   5.94524*10^-11}, {0.60206, 5.99786*10^-11}, {0.69897,
  5.97362*10^-11}, {0.778151, 5.91306*10^-11}, {0.845098,
  5.83566*10^-11}, {0.90309, 5.75112*10^-11}, {0.954243,
  5.66447*10^-11}}
\end{verbatim}

Now that we know $a=25$ provides a small enough error, we plot, once again, the PDF for various values of $s$, as shown in Figure (\ref{fig: plot4}). We quickly see that the PDF does not converge to $1$ and actually changes for each value of $s$. Even though we only take our sum out to $k=\pm 50$, this is enough to suggest that Benford behavior is unlikely.

\begin{figure}[h]
\includegraphics[width=0.6\textwidth, height= 0.43\textwidth]{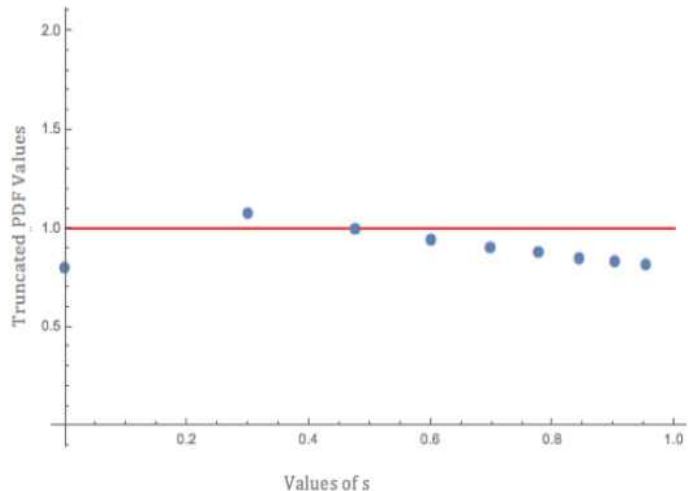}
\caption{$U \ {\rm and} \ V \sim {\rm Exp[1]}$}\label{fig: plot4}
\end{figure}
\FloatBarrier

\ \\

\noindent\textbf{Checking the Marginals}

To understand why this might be the case, we took a look at the marginal distributions. We note that $X = 10^U$ where $U \sim N[0,1]$ is a closely Benford distribution with $\chi^{2} \approx 0.9918$, but $Y = 10^V$ where $V \sim {\rm Exp[1]}$ is not, with $\chi^{2} \approx 0.7084$. Thus, in the independent case we would expect that two variables modeled like $X$, or any product with $X$, should yield a Benford distribution. The product of two variables modeled like $Y$, however, should not be Benford.


\ \\

\end{document}